\documentclass[11pt]{article}

\usepackage{amsmath,amsfonts,amssymb} 

\usepackage{amsthm} 

\usepackage{graphicx,caption,subcaption} 

\usepackage{tikz}
\tikzstyle{vtx} = [circle, fill, inner sep=1.5]
\tikzstyle{arch} = [out=30, in=150]
\tikzstyle{rlabel} = [midway, inner sep=0, outer sep=0,circle,fill=white, node contents={\tiny R}]
\tikzstyle{blabel} = [midway, inner sep=0, fill=white, node contents={\tiny B}]

\usetikzlibrary{decorations.pathreplacing}

\usepackage{fullpage} 
\usepackage[colorlinks]{hyperref} 
\usepackage[parfill]{parskip} 

\usepackage{color, soul} 
\setuldepth{x} 

\newtheorem{theorem}{Theorem}[section]

\newtheorem{conjecture}{Conjecture}[section]

\newtheorem{lemma}{Lemma}[section]
\newtheorem{claim}{Claim}[section]

\theoremstyle{definition}

\theoremstyle{remark}

\title{Online Ramsey numbers of ordered paths and cycles}

\author{%
Felix Christian Clemen\footnote {Department of Mathematics, Karlsruhe Institute of Technology, 76131 Karlsruhe, Germany, Email: \texttt{felix.clemen@kit.edu}}%
 \and Emily Heath\footnote {Department of Mathematics, Iowa State University, Ames, Iowa 50011, USA, Email: \texttt{eheath@iastate.edu}. Research partially supported by NSF RTG Grant DMS-1839918.}%
 \and Mikhail Lavrov\footnote {Department of Mathematics, Kennesaw State University, Marietta, Georgia 30067, USA, Email: \texttt{mlavrov@kennesaw.edu}}%
  }

\begin{document}

\maketitle

\begin{abstract}
    An ordered graph is a graph with a linear ordering on its vertices. The online Ramsey game for ordered graphs $G$ and $H$ is played on an infinite sequence of vertices; on each turn, Builder draws an edge between two vertices, and Painter colors it red or blue. Builder tries to create a red $G$ or a blue $H$ as quickly as possible, while Painter wants the opposite. The online ordered Ramsey number $r_o(G,H)$ is the number of turns the game lasts with optimal play. 
    
    In this paper, we consider the behavior of $r_o(G,P_n)$ for fixed $G$, where $P_n$ is the monotone ordered path. We prove an $O(n \log_2n)$ bound on $r_o(G,P_n)$ for all $G$ and an $O(n)$ bound when $G$ is $3$-ichromatic; we partially classify graphs $G$ with $r_o(G,P_n) = n + O(1)$. Many of these results extend to $r_o(G,C_n)$, where $C_n$ is an ordered cycle obtained from $P_n$ by adding one edge.
\end{abstract}

\section{Introduction}

The \emph{Ramsey number} $r(G,H)$ of two graphs $G$ and $H$ is the least $n$ such that any red-blue edge-coloring of the complete graph $K_n$ contains a red copy of $G$ or a blue copy of $H$. Finding bounds on such Ramsey numbers has been an important problem in graph theory for many years; see~\cite{CFS} for a survey of known bounds on Ramsey numbers.

The online version of the Ramsey problem considers a setting where the red-blue coloring is not given at once, but is revealed gradually. To consider the worst-case scenario, we model this setting as a game between two players, Builder and Painter, on an infinitely large set of vertices. Again, we pick two graphs $G$ and $H$.  On each turn, Builder draws an edge between two vertices and Painter colors it red or blue. The game ends when a red copy of $G$ or a blue copy of $H$ is formed; Builder seeks to minimize the number of turns, and Painter seeks to maximize it. 

The \emph{online Ramsey number} is the number of rounds in this game, assuming optimal play. Equivalently, it is the minimum number of edge queries necessary to find a red copy of $G$ or a blue copy of $H$ in an infinite complete graph, assuming a worst-case scenario for the outcome of these queries. Online Ramsey numbers are also called \emph{online size Ramsey numbers} in the literature, since they track the size (number of edges) used. Introduced independently by~\cite{B3} and~\cite{KR}, these numbers have since been studied by many authors including~\cite{AB},~\cite{C},~\cite{CDLL}, and~\cite{GKP}.

\subsection{Ordered graphs}

In this paper, we consider online Ramsey numbers of ordered graphs. An \emph{ordered graph} is a graph with a linear ordering on its vertices. We think of the vertices of an ordered graph as being arranged on a horizontal line in order from left to right. Subgraphs of an ordered graph $G$ inherit the ordering on $V(G)$, and isomorphisms between ordered graphs must be order-preserving.

From now on, we assume that all graphs are ordered graphs. We will borrow the notation $P_n$, $C_n$, $K_n$, $K_{n_1, \dots, n_k}$ to use for ordered graphs, by giving their vertices a standard ordering:
\begin{itemize}
    \item The ordered path $P_n$ has $n$ vertices $v_1 < v_2 < \dots < v_n$ and edges $v_i v_{i+1}$ for $i=1, \dots, n-1$.
    \item The ordered cycle $C_n$ is obtained from $P_n$ by adding the edge $v_1 v_n$.
    \item The ordered complete graph $K_n$ has all edges between $n$ vertices $v_1 < v_2 < \dots < v_n$. (Here, all orderings of the vertices yield isomorphic ordered graphs.)
    \item The ordered complete $k$-partite graph $K_{n_1, n_2, \dots, n_k}$ has $k$ parts with $n_1, n_2, \dots, n_k$ vertices respectively; for each $i$, the $n_i$ vertices in the $i$\textsuperscript{th} part are consecutive in the vertex ordering.
\end{itemize}

Ramsey problems on ordered graphs are motivated by the Erd\H os--Szekeres theorem~\cite{ESz}. This result asserts that any sequence of $(r-1)(s-1)+1$ distinct real numbers contains either an increasing subsequence of length $r$ or a decreasing subsequence of length $s$, while a sequence of length $(r-1)(s-1)$ is not enough. 

The Erd\H os--Szekeres theorem has many proofs~\cite{MR1380525}, most of which easily extend to the stronger result that the ordered Ramsey number of $P_r$ versus $P_s$ is $(r-1)(s-1)+1$.
Given a sequence of distinct real numbers $x_1, x_2, \dots, x_n$, we color $K_n$ by coloring $v_i v_j$ red if $x_i < x_j$, and blue if $x_i > x_j$. A red $P_r$ or blue $P_s$ in $K_n$ corresponds to an increasing or decreasing subsequence of the desired length. Not all colorings of $K_n$ originate from a sequence $x_1, x_2, \dots, x_n$, but this turns out not to affect the result.

We will write $r_o(G,H)$ for the online Ramsey number for ordered graphs $G$ and $H$. These were  first studied in the case $r_o(P_r, P_s)$ by~\cite{BCHL2} and~\cite{PPW}. Just as in the case of unordered graphs, the online Ramsey number is defined by a game played between Builder and Painter on an infinite set of vertices. Here, we assume that the vertex set is $\mathbb N$ with the usual ordering.

It will sometimes be convenient, when describing a strategy for Builder, to assume that between any two vertices that have already been used, another vertex can be found. This can be guaranteed even on vertex set $\mathbb N$ with foresight on Builder's part: if Builder's strategy wins in $t$ turns, then at most $2t$ distinct vertices are used, so Builder may choose the $i^{\text{th}}$ vertex to be an element of $\mathbb N$ divisible by $2^{2t-i}$ but not by $2^{2t-i+1}$. This will be needed, for example, in the proof of Theorem~\ref{thm:3-ichromatic}.

The online Ramsey game for ordered graphs is closely related to a sorting problem: given an infinite sequence $x_1, x_2, \dots$ of distinct real numbers, how many comparisons are needed to locate an increasing subsequence of length $r$ or a decreasing subsequence of length $s$? However, Painter has more flexibility in coloring edges, since edge colors do not have to obey transitivity; it is unknown whether the two problems are equivalent.

\subsection{Our results}

In this paper, we consider online ordered Ramsey numbers of the form $r_o(G, P_n)$ for various ordered graphs $G$. We are especially interested in how $r_o(G, P_n)$ varies with $n$ when $G$ is fixed.

Theorem 2.1 in~\cite{BCHL2} gives an upper bound on $r_o(P_m, P_n)$ of the form $O(mn \log_2n)$. Our first result is a generalization of this upper bound to $r_o(G, P_n)$. Here, let $\Delta^-(G)$ denote the \emph{maximum left degree} of $G$: the maximum number of edges between any vertex $v$ and vertices that precede $v$. 

\begin{theorem}\label{thm:left-degree}
    For any ordered graph $G$, $r_o(G, P_n) \le \Delta^-(G) |V(G)| n \log_2 n$.
\end{theorem}

By symmetry, the same bound applies with left degree replaced by right degree.

Note that in Theorem 2.1 of~\cite{BCHL2}, we are free to swap $m$ and $n$, so for fixed $m$, when $n$ is large, it also gives an upper bound of $O(mn \log_2m)$ which is linear in $n$. However, Theorem~\ref{thm:left-degree} does not allow the same flexibility. This motivates the question: for which $G$ is $r_o(G, P_n)$ linear in $n$, when $G$ is fixed?

An \emph{interval coloring} of an ordered graph $G$ is a proper vertex coloring of $G$ in which each color class is a set of consecutive vertices: if $u,w \in V(G)$ are assigned the same color, then every vertex $v$ between $u$ and $w$ must also be assigned that color. The \emph{interval chromatic number} $\chi_i(G)$ of an ordered graph $G$ is the minimum number of colors in an interval coloring of $G$. We say that $G$ is \emph{$k$-ichromatic} if $\chi_i(G) \le k$.

The interval chromatic number has played a key role in Ramsey and Tur\'an problems for ordered graphs; see for example \cite{MR3575208}, \cite{GHLNPW}, \cite{MR2269435} and \cite{MR4003657}. We are able to prove a linear upper bound on $r_o(G, P_n)$ for any $3$-ichromatic $G$.

\begin{theorem}\label{thm:3-ichromatic}
    For any ordered graph $G$ such that $\chi_i(G) \le 3$, $r_o(G,P_n) = O(n |V(G)|^2 \log_2|V(G)|)$. 
\end{theorem}

It is certainly not the case that the interval chromatic number tells the whole story: for example, $r_o(P_k, P_n) = O(n)$ for any fixed $k$ even though $\chi_i(P_k) = k$. On the other hand, we do not have a linear upper bound even on $r_o(C_4, P_n)$. 

When we look at smaller and sparser graphs $G$, another transition emerges. Let $M_k$ denote the \emph{serial $k$-edge matching}: the ordered graph with $E(M_k) = \{v_1w_1, v_2w_2, \dots, v_k w_k\}$ and $v_1 < w_1 < v_2 < w_2 < \dots < v_k < w_k$. 

\begin{theorem}\label{thm:serial-matching}
For $n \ge 4$ and $k \ge 2$, $r_o(M_k, P_n) \le n + 2k-4$, and if $n \ge \max\{3k-3,2k+1 \}$, then $r_o(M_k, P_n) = n+2k-4$.
\end{theorem}

On the other hand, we have stronger lower bounds on $r_o(G,P_n)$ even in the case of three two-edge graphs: the path $P_3$, the intersecting matching $X$, and the $2$-pronged claw $K_{1,2}$ (shown in Figure~\ref{fig:two-edge}).

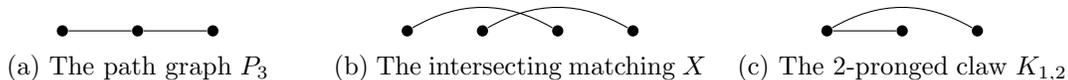
\begin{figure}[h!]
    \centering
    \begin{subfigure}{0.3\textwidth}
    \centering
    \begin{tikzpicture}
        \node [vtx] (0) at (0,0) {};
        \node [vtx] (1) at (1,0) {};
        \node [vtx] (2) at (2,0) {};
        \draw (0) -- (1) -- (2);
        \path (0) to [arch] (2); 
    \end{tikzpicture}
    \caption{The path graph $P_3$}
    \end{subfigure}
    \begin{subfigure}{0.3\textwidth}
    \centering
    \begin{tikzpicture}
        \node [vtx] (0) at (0,0) {};
        \node [vtx] (1) at (1,0) {};
        \node [vtx] (2) at (2,0) {};
        \node [vtx] (3) at (3,0) {};
        \draw (0) to [arch] (2) (1) to [arch] (3);
    \end{tikzpicture}
    \caption{The intersecting matching $X$}
    \end{subfigure}
    \begin{subfigure}{0.3\textwidth}
    \centering
    \begin{tikzpicture}
        \node [vtx] (0) at (0,0) {};
        \node [vtx] (1) at (1,0) {};
        \node [vtx] (2) at (2,0) {};
        \draw (0) -- (1) (0) to [arch] (2);
    \end{tikzpicture}
    \caption{The $2$-pronged claw $K_{1,2}$}
    \end{subfigure}
    \caption{The three two-edge graphs of Theorem~\ref{thm:bad-small-graphs}}
    \label{fig:two-edge}
\end{figure}

\newpage
\begin{theorem}\label{thm:bad-small-graphs}
The following bounds hold:
\begin{enumerate}
    \item[(a)] $2n-2 \le r_o(P_3, P_n) \le \frac{8}{3}n - \frac{10}{3}$;
    \item[(b)] $\frac32 n -\frac32< r_o(X, P_n) \leq \frac32n + 2$; 
    \item[(c)] $r_o(K_{1,2}, P_n) = 2n$ (and by symmetry, $r_o(K_{2,1}, P_n) = 2n$).
\end{enumerate}
\end{theorem}

In particular, $r_o(G, P_n) \ge \frac32n$ for any $G$ that contains any of $P_3$, $X$, $K_{1,2}$, or $K_{2,1}$ as a subgraph. If $G$ does not contain any of $P_3$, $K_{1,2}$, and $K_{2,1}$, then $G$ must be a matching (possibly with isolated vertices, which don't affect $r_o(G,P_n)$). We say that an ordered matching $M$ is \emph{intersection-free} if it also does not contain $X$.

We conjecture that having one of the two-edge graphs in Theorem~\ref{thm:bad-small-graphs} as a subgraph is the only reason why a bound of the form $r_o(G,P_n) = n + O(1)$ would not hold.

\begin{conjecture}\label{st-ives-conjecture}
For every intersection-free matching $M$, there is a constant $c$ such that $r_o(M, P_n) \le n + c$.
\end{conjecture}

We can rephrase Conjecture~\ref{st-ives-conjecture} in terms of a more concrete sequence of Ramsey problems. Let the \emph{St.~Ives matching}\footnote{We named this ordered matching after the town St.~Ives from the rhyme ``As I was going to St.~Ives,'' which, in an early version, first appeared in 1730 in a manuscript by Harley.} $S_k$ be the ordered matching defined recursively as follows. $S_0$ is simply an ordered edge. $S_k$ is constructed from two disjoint consecutive copies of $S_{k-1}$ by adding an edge $u_kv_k$ such that both copies of $S_{k-1}$ are to the right of $u_k$ and to the left of $v_k$. A diagram of $S_2$ is shown in Figure~\ref{subfig:s2}.

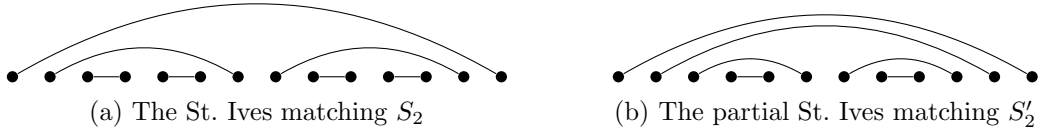
\begin{figure}[h]
    \centering
    \begin{subfigure}{0.45\textwidth}
    \centering
    \begin{tikzpicture}[scale=0.5]
		\foreach \i in {1,...,14} {\node [vtx] (\i) at (\i,0) {};}
		\draw (3) -- (4) (5) -- (6) (9) -- (10) (11) -- (12);
		\draw (2) to [arch] (7) (8) to [arch] (13) (1) to [arch] (14);
    \end{tikzpicture}
    \caption{The St.~Ives matching $S_2$}
	\label{subfig:s2}
    \end{subfigure}
    \begin{subfigure}{0.45\textwidth}
    \centering
    \begin{tikzpicture}[scale=0.5]
		\foreach \i in {1,...,12} {\node [vtx] (\i) at (\i,0) {};}
		\draw (4) -- (5) (8) -- (9);
		\draw (3) to [arch] (6) (7) to [arch] (10) (2) to [arch] (11) (1) to [arch] (12);
    \end{tikzpicture}
    \caption{The partial St.~Ives matching $S_2'$}
	\label{subfig:s2'}
    \end{subfigure}
    \caption{Examples of the St.~Ives matching and the partial St.~Ives matching}
    \label{fig:st-ives}
\end{figure}

Every intersection-free matching $M$ on $k$ edges is a subgraph of $S_k$, which can be shown recursively. If $M$ splits into two matchings $M'$, $M''$ where every vertex of $M'$ lies before every vertex of $M''$, we can find $M$ and $M'$ within the two different copies of $S_{k-1}$ inside $S_k$. If $M$ does not split in this way, then it has an edge $vw$ where $v$ is the leftmost vertex of $M$ and $w$ is the rightmost vertex. Then we can find a copy of $M-vw$ inside one copy of $S_{k-1}$ inside $S_k$; together with edge $u_kv_k$ of $S_k$, it becomes a subgraph isomorphic to $M$. 

As a result, Conjecture~\ref{st-ives-conjecture} is equivalent to the following statement.

\begin{conjecture}\label{st-ives-conjecture-v2}
For all $k \ge 1$, there is a constant $c_k$ such that for all $n\ge 1$, $r_o(S_k, P_n) \le n + c_k$.
\end{conjecture}

In partial support of Conjecture~\ref{st-ives-conjecture-v2}, we show it for a less general sequence of graphs. First, let the \emph{nested matching} $N_k$ be the ordered graph with vertices $v_1 < \dots < v_k < w_k < \dots < w_1$ and edges $v_1w_1, \dots, v_kw_k$. The \emph{partial St.~Ives matching} $S_k'$ is the ordered matching obtained from $N_k$ by adding $k$ more consecutive copies of $N_k$ between vertices $v_k$ and $w_k$. A diagram of $S_2'$ is shown in Figure~\ref{subfig:s2'}. 

\begin{theorem}\label{thm:partial-st-ives}
For all $k \ge 1$, there is a constant $c_k$ such that for all $n\ge 1$, $r_o(S_k', P_n) \le n + c_k$.
\end{theorem}

In fact, Builder's strategy in the proof of Theorem~\ref{thm:partial-st-ives} can be repeated to extend a path in two directions, obtaining either a blue $P_n$ or two consecutive copies of $S_k'$. The graph consisting of three consecutive copies of $S_k'$ is the simplest matching for which we are unable to prove that Conjecture~\ref{st-ives-conjecture} holds.

All of the results above are bounds on $r_o(G, P_n)$. However, when we replace $P_n$ by the ordered cycle $C_n$, the same classification into $O(n\log_2n)$ and $O(n)$ upper bounds holds, due to the following result.

\begin{theorem}\label{thm:path-to-cycle}
    For any ordered graph $G$, there is a constant $c$ such that $r_o(G, C_n) \le r_o(G, P_m) + c$, where $m = (n-1)(|V(G)|-1) + 1$.
\end{theorem}

All of our bounds are proved by giving deterministic strategies for Builder and Painter.

The remainder of the paper is organized as follows. In Section~\ref{section:general}, we prove $O(n\log_2n)$ bounds on $r_o(G,P_n)$ and $r_o(G,C_n)$ that apply to any ordered graph $G$: Theorem~\ref{thm:left-degree} and Theorem~\ref{thm:path-to-cycle}. In Section~\ref{section:3-ichromatic}, we prove Theorem~\ref{thm:3-ichromatic}, which applies to $3$-ichromatic graphs $G$, for which we can prove $O(n)$ bounds. In Section~\ref{section:small-graphs}, we prove the three cases of Theorem~\ref{thm:bad-small-graphs}, dealing with small graphs for which we can still do no better than $O(n)$. Finally, in Section~\ref{section:n+O(1)}, we prove Theorem~\ref{thm:serial-matching} and Theorem~\ref{thm:partial-st-ives}: two cases in which $r_o(G,P_n) = n + O(1)$.
\section[Results for general G]{Results for general $G$}
\label{section:general}

\subsection[An O(n log n) bound]{An $O(n\log_2n)$ bound}

\begin{proof}[Proof of Theorem~\ref{thm:left-degree}]
Builder maintains a list of graphs $G_1, G_2, \dots, G_{n-1}$, where:
\begin{itemize}
\item Each $G_i$ is a monochromatic red subgraph of the graph built so far.
\item Each $G_i$ is isomorphic to an ``initial subgraph of $G$'': the subgraph obtained by taking the leftmost $v(G_i)$ vertices of $G$.
\item Each vertex of $G_i$ is the rightmost endpoint of a blue $P_i$.
\end{itemize}
Initially, every $G_i$ is the null graph (with $0$ vertices). Figure~\ref{fig:left-degree-graphs} shows an example of an intermediate stage of Builder's strategy, with graphs $G_1, G_2, G_3, G_4$ which are all subgraphs of $G = K_4$.

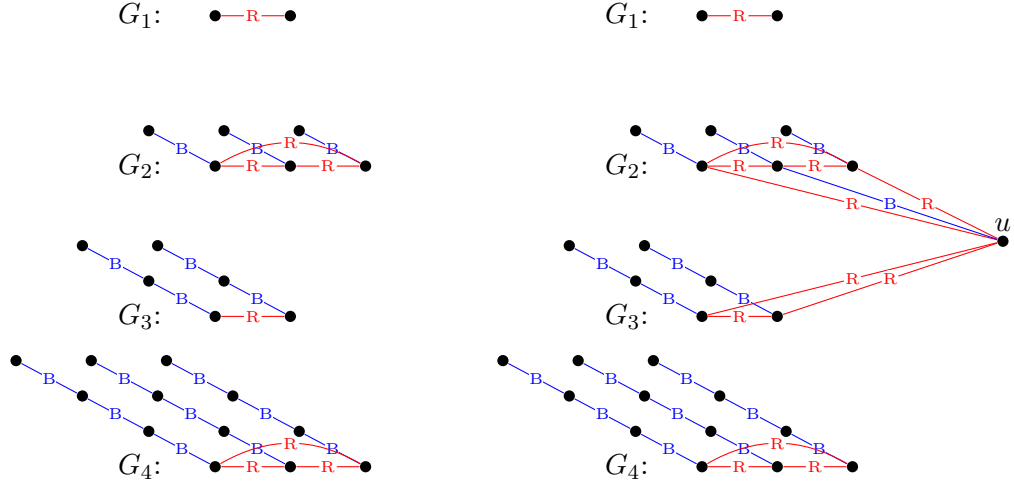
\begin{figure}
	\centering
	\begin{subfigure}{0.45\textwidth}
	\centering
	\begin{tikzpicture}
		\foreach \i in {0,1,2} {
			\foreach \j in {0,1,2,3} {
				\node [vtx] (a\i\j) at (\i-15/17*\j, 8/17*\j) {};
			}
			\foreach \j [count=\jcount] in {0,1,2} {
				\draw [blue] (a\i\jcount) to node[blabel]{} (a\i\j);
			} 
		}
		\foreach \i in {0,1} {
			\foreach \j in {0,1,2} {
				\node [vtx] (b\i\j) at (\i-15/17*\j, 8/17*\j+2) {};
			}
			\foreach \j [count=\jcount] in {0,1} {
				\draw [blue] (b\i\jcount) to node[blabel]{} (b\i\j);
			} 
		}
		\foreach \i in {0,1,2} {
			\foreach \j in {0,1} {
				\node [vtx] (c\i\j) at (\i-15/17*\j, 8/17*\j+4) {};
			}
			\draw [blue] (c\i1) to node[blabel]{} (c\i0);
		}
		\node [vtx] (d0) at (0,6) {}; \node [vtx] (d1) at (1,6) {};
		\draw [red] (a00) to node[rlabel]{} (a10) to node[rlabel]{} (a20) (a00) to[arch] node[rlabel]{} (a20);
		\draw [red] (b00) to node[rlabel]{} (b10);
		\draw [red] (c00) to node[rlabel]{} (c10) to node[rlabel]{} (c20) (c00) to[arch] node[rlabel]{} (c20);
		\draw [red] (d0) to node[rlabel]{} (d1);
		
		\node at (-1,6) {$G_1$:};
		\node at (-1,4) {$G_2$:};
		\node at (-1,2) {$G_3$:};
		\node at (-1,0) {$G_4$:};
	\end{tikzpicture}
	\caption{An example of graphs $G_1, G_2, G_3, G_4$ and \\ the required blue paths}
	\label{fig:left-degree-graphs}
	\end{subfigure}
	\begin{subfigure}{0.45\textwidth}
	\centering
	\begin{tikzpicture}
		\foreach \i in {0,1,2} {
			\foreach \j in {0,1,2,3} {
				\node [vtx] (a\i\j) at (\i-15/17*\j, 8/17*\j) {};
			}
			\foreach \j [count=\jcount] in {0,1,2} {
				\draw [blue] (a\i\jcount) to node[blabel]{} (a\i\j);
			} 
		}
		\foreach \i in {0,1} {
			\foreach \j in {0,1,2} {
				\node [vtx] (b\i\j) at (\i-15/17*\j, 8/17*\j+2) {};
			}
			\foreach \j [count=\jcount] in {0,1} {
				\draw [blue] (b\i\jcount) to node[blabel]{} (b\i\j);
			} 
		}
		\foreach \i in {0,1,2} {
			\foreach \j in {0,1} {
				\node [vtx] (c\i\j) at (\i-15/17*\j, 8/17*\j+4) {};
			}
			\draw [blue] (c\i1) to node[blabel]{} (c\i0);
		}
		\node [vtx] (d0) at (0,6) {}; \node [vtx] (d1) at (1,6) {};

		\draw [red] (a00) to node[rlabel]{} (a10) to node[rlabel]{} (a20) (a00) to[arch] node[rlabel]{} (a20);
		\draw [red] (b00) to node[rlabel]{} (b10);
		\draw [red] (c00) to node[rlabel]{} (c10) to node[rlabel]{} (c20) (c00) to[arch] node[rlabel]{} (c20);
		\draw [red] (d0) to node[rlabel]{} (d1);
		
		\node [vtx] (u) at (4,3) {};
		\draw [red] (c00) to node[rlabel]{} (u);
		\draw [blue] (c10) to node[blabel]{} (u);
		\draw [red] (c20) to node[rlabel]{} (u);
		\draw [red] (b00) to node[rlabel]{} (u);
		\draw [red] (b10) to node[rlabel]{} (u);

		\node at (-1,6) {$G_1$:};
		\node at (-1,4) {$G_2$:};
		\node at (-1,2) {$G_3$:};
		\node at (-1,0) {$G_4$:};
		\node [above] at (u) {$u$};
	\end{tikzpicture}
	\caption{Comparing $u$ to $G_2$ and $G_3$: the result is $G_2 < u$ and $u \le G_3$, so $u$ can be added to $G_3$}
	\label{fig:left-degree-comparing-u}
	\end{subfigure}
	\caption{Builder's strategy for Theorem~\ref{thm:left-degree}, where $G = K_4$. The vertical positions of vertices are only varied for clarity; in fact, the vertices on different blue paths or in different $G_i$ are not required to be in any particular order relative to each other, or even to be distinct.}
	\label{fig:left-degree}
\end{figure}

Builder performs many iterations of adding a new vertex to some $G_i$. The new vertex $u$ must be to the right of every previously considered vertex.

Each iteration consists of many steps we'll call ``comparing $u$ to $G_i$". To compare $u$ to $G_i$, Builder draws all the edges between $G_i$ and $u$ that make $G_i \cup \{u\}$ isomorphic to an initial subgraph of $v(G_i)+1$ vertices of $G$; an example is shown in Figure~\ref{fig:left-degree-comparing-u}. We say that the result of the comparison is ``$G_i < u$" if one of these edges is blue, and ``$u \le G_i$" if all of these edges are red. (The special case where $G_i$ is a null graph may not be immediately clear; in this case, there are no edges to check, and we always have $u\le G_i$.) Note that each comparison requires drawing at most $\Delta^-(G)$ edges.

In $\lceil \log_2 n \rceil$ comparisons, Builder can find an $i$ such that $G_{i-1} < u \le G_i$, in which case $u$ can be added to $G_i$. (Since $G_{i-1} < u$, $u$ is the leftmost vertex of a blue $P_i$; since $u \le G_i$, $G_i \cup \{u\}$ is monochromatic red.) As a special case, if $u \le G_1$, then $u$ can be added to $G_1$. As another special case, if $G_{n-1} < u$, then $u$ is the leftmost endpoint of a blue $P_n$, and Builder wins immediately.

After $(|V(G)|-1)(n-1)+1$ iterations, if Builder does not win earlier, then some $G_i$ will have $|V(G)|$ vertices, and there is a red copy of $G$.
\end{proof}

Thus for any fixed graph $G$, we have $r_o(G, P_n)=O(n \log_2n)$.

\subsection{Replacing the path by the cycle}

In this section, we prove Theorem~\ref{thm:path-to-cycle} that $r_o(G, C_n)$ may be bounded in terms of $r_o(G, P_m)$ up to a small additive constant. In fact, we can make the following more precise claim. For all $k \ge 1$, there is a constant $c_k$ such that if $G$ is any $k$-vertex ordered graph, then $r_o(G, C_n) \le r_o(G, P_{kn-n-k+2}) + c_k$.

To prove this theorem, we begin with three claims about steps in Builder's strategy. There is a common idea in all three proofs: if Builder draws an edge $vw$ between two vertices of a blue cycle, and Painter colors it blue, then a shorter blue cycle is created by skipping all vertices of the original cycle between $v$ and $w$. We say that such an edge $vw$ has \emph{length} $\ell$ if there are $\ell-1$ vertices of the cycle between $v$ and $w$. 

To force Painter to color such edges blue, Builder may draw a copy of $G$ using vertices of the cycle. We say that such a copy of $G$ is \emph{scaled by} $\ell$ if there are $\ell-1$ vertices of the cycle between any two consecutive vertices of $G$. Edges in such a copy of $G$ can have length $\ell, 2\ell, \dots, (k-1)\ell$. We will assume that Painter never colors a copy of $G$ entirely red, because then Builder would win.

Builder may draw multiple scaled copies of $G$ on the same cycle to obtain multiple blue edges, before skipping any vertices. In such a case, each copy of $G$ is drawn to the right of all previous copies, so that all of the blue edges (or only some of them) can be used to skip vertices together.

\begin{claim}\label{p2c-linear-to-epsilon}
	For all $\epsilon>0$, if $n$ is sufficiently large, then starting from a blue ordered cycle with length between $(1+2\epsilon)n$ and $kn$, Builder can force either a blue ordered cycle with length between $(1+\epsilon)n$ and $(1+2\epsilon)n$ or a red $G$ in $O(1)$ steps.
\end{claim}
\begin{proof}
Suppose that there is a blue $C_{(1+\epsilon)n + t}$ for some $t \ge \epsilon n$. Builder draws a copy of $G$ scaled by $\left\lfloor\frac{t}{k-1}\right\rfloor$ on the cycle, obtaining a blue edge of length at least $\left\lfloor\frac{t}{k-1}\right\rfloor$ and at most $t$. Builder uses this edge to obtain a cycle with length at least $(1+\epsilon)n$ and at most $(1+\epsilon)n + t - \left\lfloor\frac{t}{k-1}\right\rfloor$. For sufficiently large $n$, this upper bound is at most $(1+\epsilon)n + (1 - \frac1k)t$.

After repeating this procedure up to $j$ times, Builder obtains a cycle with length at least $(1+\epsilon)n$ and at most $(1+\epsilon)n + (1-\frac1k)^j t$. For $j \ge \frac{\log_2(k/\epsilon)}{\log_2(k/(k-1))}$, the length of this cycle is guaranteed to be in the range we want.
\end{proof}

\begin{claim}\label{p2c-epsilon-to-constant}
	If $n$ is sufficiently large, then starting from a blue ordered cycle with length between $(1+\frac1{2k^2})n$ and $(1+ \frac1{k^2})n$, Builder can force either a blue $C_{n+t}$ with $k^3 \le t \le k^3 + 2(k-1)!$ or a red $G$ in $O(1)$ steps.
\end{claim}
\begin{proof}
Let the given blue cycle have length $C_{n+k^3+\ell}$.

We assume that $n \ge 2k^5 + 2k^2(k-1)!$, so that $\ell \ge (k-1)!$. This means that Builder can draw copies of $G$ scaled by $\lfloor \ell/(k-1)!\rfloor$ on the cycle. Builder draws $k!$ such copies of $G$; each one occupies fewer than $k\ell/(k-1)!$ vertices, for at most $k^2\ell < n$ vertices.

For some $j \le k-1$, at least $(k-1)!$ of these copies of $G$ give blue edges of length $j\lfloor \ell/(k-1)!\rfloor$. Builder uses $(k-1)!/j$ of these edges to shorten the cycle, removing $\frac{(k-1)!}{j}(j\lfloor \ell/(k-1)!\rfloor - 1) = (k-1)!\lfloor \ell/(k-1)!\rfloor - (k-1)!/j$ vertices total: between $\ell$ and $\ell - 2(k-1)!$. 

The result is a cycle of length between $n+k^3$ and $n+k^3 + 2(k-1)!$.
\end{proof}

\begin{claim}\label{p2c-constant-to-done}
	If $n$ is sufficiently large, then starting from a blue $C_{n+t}$ with $n/k 
 \ge t \ge k^3$, Builder can force either a blue $C_n$ or a red $G$ in $O(t)$ steps.
\end{claim}
\begin{proof}
Builder begins by drawing $kt$ copies of $G$ scaled by $1$ (that is, using consecutive vertices of $C_{n+t}$). For some $j<k$, at least $t$ of these copies give an edge of length $j$; let $E_1$ be a set of $t$ such edges. Before using any of these edges, Builder continues by drawing $k$ copies of $G$ scaled by $j-1$. Each of these gives an edge whose length is a multiple of $j-1$; let $E_2$ be the set of all $k$ of these edges.

Each edge in $E_2$ can be used to decrease the length of the cycle by $i(j-1)-1$ for some $1 \le i \le k-1$; this is at most $k^2$, and one less than a multiple of $j-1$. Builder uses enough edges from $E_2$ to shorten the cycle to length $C_{n+t'}$ where $t' \equiv 0 \pmod{j-1}$; since $t \ge k^3$, $t' \ge 0$. Finally, Builder uses $t'/(j-1)$ of the edges from $E_1$ to shorten the cycle to length exactly $n$.
\end{proof}

\begin{proof}[Proof of Theorem~\ref{thm:path-to-cycle}]

In cases where $n$ is not large enough to apply Claim~\ref{p2c-linear-to-epsilon}, Claim~\ref{p2c-epsilon-to-constant}, or Claim~\ref{p2c-constant-to-done}, then $n$ is bounded by a function of $k$, and we can handle all such cases by choosing $c_k$ large enough.

Otherwise, Builder first uses the $r_o(G, P_{nk-n-k+2})$ strategy to force either a red $G$ (and win) or a blue $P_{nk-n-k+2}$, whose vertices we will label $0, 1, 2, \dots, (k-1)(n-1)$. Then, Builder draws $G$ at vertices $0, n-1, 2(n-1), \dots, (k-1)(n-1)$.

If all edges of this copy of $G$ are red, then Builder immediately wins. Otherwise, there is a blue edge between vertices $i(n-1)$ and $j(n-1)$ for some $i<j$.

If $j=i+1$, then there is a blue $C_n$ and Builder also immediately wins. Otherwise, one of the other edges drawn in the second step is red, and there is a blue cycle with at least $2n-1$ and at most $kn-n-k+2$ vertices. 

In this case:
\begin{enumerate}
\item Builder executes the strategy of Claim~\ref{p2c-linear-to-epsilon}, taking $\epsilon = \frac1{2k^2}$, obtaining a blue ordered cycle with length between $(1+\frac2{k^2})n$ and $(1+ \frac1{k^2})n$.
\item Builder executes the strategy of Claim~\ref{p2c-epsilon-to-constant}, obtaining a blue ordered cycle with length between $n + k^3$ and $n + k^3 + 2(k-1)!$.
\item Builder executes the strategy of Claim~\ref{p2c-constant-to-done}; since $2(k-1)! = O(1)$, this also takes $O(1)$ steps.
\end{enumerate}
This procedure ends with either a blue $C_n$ or a red $G$.
\end{proof}

\textit{Remark:} We make no effort to optimize the constant $c_k$ in this proof. The length of $P_{nk-n-k+2}$ can also be improved in some cases; this exact length is only required in the first step of the proof of the theorem. In particular, if $G$ is $2$-ichromatic, then a path of length $n + k^3 + k$ is enough. Builder can draw $G$ using the first and last vertices of the path, such that all its edges have length at least $n + k^3$, and then skip directly to applying Claim~\ref{p2c-constant-to-done}.

\section[Results for 3-ichromatic graphs]{Results for $3$-ichromatic graphs}
\label{section:3-ichromatic}

The results in the previous section bound $r_o(G,P_n)$ for general ordered graphs $G$. In this section, we show improved bounds when $G$ is $3$-ichromatic. Equivalently, $G$ is a subgraph of $K_{a,b,c}$, an ordered complete tripartite graph.

\begin{lemma}\label{lemma:tripartite}
For all $a,b,c,d \ge 1$ and $n \ge a + b + c + 2d$, suppose $G$ is a subgraph of $K_{a,b,c}$ with $|E(G)| = m$. Then
\[
    r_o(G, P_n) \le n \cdot \frac2d \cdot(m + r_o(G, P_{a+b+c+2d})).
\]
\end{lemma}
\begin{proof}
For convenience, let $R = r_o(G, P_{a+b+c+2d})$; we will eventually use Theorem~\ref{thm:left-degree} to bound $R$. To keep track of Builder's progress, we say that the graph built so far is in state $(x,y)$ when it contains a blue $P_x$ followed by a blue $P_y$, where $\min\{x,y\} \ge a+c$.

At the beginning of the game, Builder takes $R$ moves to create either a red $G$ (and win) or a blue $P_{a+b+c+2d}$; then, Builder takes $R$ more moves to do this again to the right of the previous $P_{a+b+c+2d}$. This results in state $(a+b+c+2d,a+b+c+2d)$.

In state $(x,y)$, Builder first takes $R$ moves to create either a red $G$ (and win) or a blue $P_{a+b+c+2d}$ between the blue $P_x$ and the blue $P_y$. Then, define sets $A,B,C$ as follows:
\begin{itemize}
	\item $A$: the last $a$ vertices of the blue $P_x$.
	\item $B$: vertices $1+c+d$ through $b+c+d$ of the blue $P_{a+b+c+2d}$.
	\item $C$: the first $c$ vertices of the blue $P_y$.
\end{itemize}
See Figure~\ref{fig:ABC} for an illustration of this definition.

\begin{figure}[h!]
	\centering

     \begin{tikzpicture}[scale=0.8]
		\node [vtx] (x1) at (-2,0) {}; 
		\node [vtx] (x2) at (-1,0) {}; 
		\node [vtx] (x3) at (0,0) {}; 
		\node [vtx] (x4) at (1,0) {};
		\node [vtx] (z2) at (3,0) {}; 
		\node [vtx] (z3) at (4,0) {}; 
		\node [vtx] (z4) at (5,0) {}; 
		\node [vtx] (z5) at (6,0) {}; 
		\node [vtx] (z6) at (7,0) {}; 
		\node [vtx] (z7) at (8,0) {}; 
		\node [vtx] (z8) at (9,0) {}; 
		\node [vtx] (y2) at (11,0) {}; 
		\node [vtx] (y3) at (12 ,0) {}; 
		\node [vtx] (y4) at (13 ,0) {}; 
		\node [vtx] (y5) at (14,0) {}; 
        
        \node [below] at (-.5,-.5) {$P_x$};
        \node [below] at (6,-.5) {$P_{a+b+c+2d}$};
        \node [below] at (12.5,-.5) {$P_y$};
        
        \node [above] at (.5,.5) {$A$};
        \node [above] at (5.5,.5) {$B$};
        \node [above] at (11,.5) {$C$};
        
		\draw[blue] (x1) to node[blabel]{} (x2);
		\draw[blue] (x2) to node[blabel]{} (x3);
		\draw[blue] (x3) to node[blabel]{} (x4);
		\draw[blue] (z2) to node[blabel]{} (z3);
		\draw[blue] (z3) to node[blabel]{} (z4);
		\draw[blue] (z4) to node[blabel]{} (z5);
        \draw[blue] (z5) to node[blabel]{} (z6);
        \draw[blue] (z6) to node[blabel]{} (z7);
        \draw[blue] (z7) to node[blabel]{} (z8);
		\draw[blue] (y2) to node[blabel]{} (y3);
		\draw[blue] (y3) to node[blabel]{} (y4);
		\draw[blue] (y4) to node[blabel]{} (y5);

        \draw (.5,0) ellipse (.8cm and .5cm);
        \draw (5.5,0) ellipse (.8cm and .5cm);
        \draw (11,0) ellipse (.4cm and .4cm);

	\end{tikzpicture}

    \caption{An illustration of the definition of the sets $A,B,$ and $C$ where $a=b=2$ and $c=d=1$.}
	\label{fig:ABC}
\end{figure}
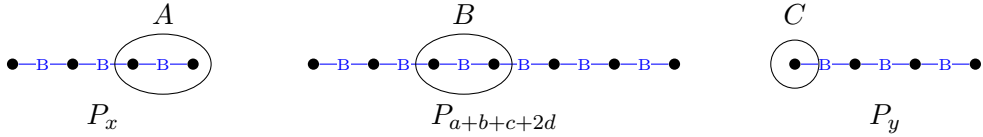

Builder takes $m$ moves to draw a copy of $G$ on vertex set $A \cup B \cup C$; all edges of $G$ have their endpoints in two different sets.

If all these edges are red, Builder wins. If there is a blue edge between $A$ and $B$, it can be used to replace the blue $P_x$ by a blue path with at least $x+d+2$ vertices, resulting in the state $(x+d+2,y)$. If there is a blue edge between $B$ and $C$, it can be used to replace the blue $P_y$ by a blue path with at least $y+d+2$ vertices, resulting in the state $(x,y+d+2)$.

Finally, if there is a blue edge between $A$ and $C$, it can be used to combine the two blue paths into a single path with at least $x+y-a-c+2$ vertices. Then, Builder takes $R$ more moves to create either a red copy of $G$ (and win) or a blue $P_{a+b+c+2d}$ to the right of this single path. This results in the state $(x+y-a-c+2, a+b+c+2d)$. (Note that $x+y-a-c+2 \ge a+c$, the minimum length we required.)

Altogether, one of two things has happened:
\begin{itemize}
\item In $R + m$ moves from state $(x,y)$, Builder obtains a state $(x',y')$ with $x'+y' = x+y+d+2$.
\item In $2R + m$ moves from state $(x,y)$, Builder obtains a state $(x',y')$ with $x'+y' = x+y+b+2d+2$.
\end{itemize}

As the game goes on, the sum $x+y$ increases at a rate of at least $d$ every $R+m$ moves, and the first $2R$ moves yield a state $(x,y)$ with $x+y > 4d$. To reach a state with $x+y \ge 2n$, it takes at most $\frac{2n}{d}-2$ repetitions of $R + m$ moves; allowing for a final sequence of $2R + m$ moves to finish, at most $\frac{2n}{d} \cdot (R + m)$ moves are required.

Therefore after Builder's strategy is followed for $\frac{2n}{d} \cdot (R + m)$ moves, $x+y$ is at least $2n$, which means that either the blue $P_x$ or the blue $P_y$ contains the $P_n$ we want. This results in a victory for Builder.
\end{proof}

\begin{proof}[Proof of Theorem~\ref{thm:3-ichromatic}]
To deduce Theorem~\ref{thm:3-ichromatic} from Lemma~\ref{lemma:tripartite}, first set $d = \frac{a+b+c}{2}$; then 
\[
    r_o(G, P_n) \le \frac{4n}{a+b+c} \left(m + r_o(G, P_{2a+2b+2c})\right) \le \frac{8n}{a+b+c} \cdot r_o(G,P_{2a+2b+2c})
\]
where the second inequality holds because $r_o(G,P_{2a+2b+2c})$ must certainly be at least $m = |E(G)|$. By Theorem~\ref{thm:left-degree}, $r_o(G, P_{2a+2b+2c}) \le \Delta^-(G) |V(G)| (2a+2b+2c) \log_2(2a+2b+2c)$. Recalling that $|V(G)|=a+b+c$, we obtain 
\[
    r_o(G,P_n) \le 16n \Delta^-(G) |V(G)| \log_2(2 |V(G)|),
\]
which is $O(n |V(G)|^2 \log_2|V(G)|)$.
\end{proof}

\section{Small graphs vs $P_n$}
\label{section:small-graphs}

In this section, we prove Theorem~\ref{thm:bad-small-graphs}, giving us bounds on $r_o(G, P_n)$ when $G$ is one of three small ``bad" graphs: the path $P_3$, the intersecting matching $X$, and the $2$-pronged claw $K_{1,2}$.

\subsection{The first bad graph: $P_3$}

To prove a lower bound on $r_o(P_3,P_n)$, we give a strategy for Painter in Lemma~\ref{lemma:p3-lower-bound}. The upper bound in Lemma~\ref{lemma:p3-upper-bound} is a result of a strategy for Builder.

\begin{lemma}\label{lemma:p3-lower-bound}
	For all $n \ge 1$, $r_o(P_3, P_n) \ge 2n-2$.
\end{lemma}
\begin{proof}
Let Painter use the following strategy: color an edge red if this does not create a red $P_3$, blue otherwise.

Each blue edge that appears in the graph must therefore be adjacent to a red edge. We call a blue edge $vw$ with $v < w$ \emph{left-forced} if $v$ has a red edge from a preceding vertex, and \emph{right-forced} if $w$ has a red edge to a following vertex. (A diagram demonstrating some left-forced and right-forced edges of a blue $P_6$ is shown in Figure~\ref{fig:p3-lower-bound}.) In either case, we say that the red edge \emph{forces} $vw$. Each blue edge must be either left-forced or right-forced, or else it could have been colored red. It's possible that a blue edge is both left-forced and right-forced.

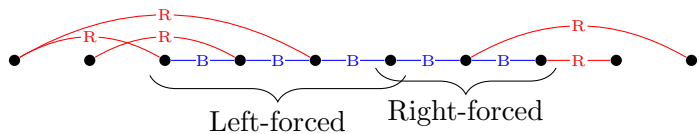
\begin{figure}[h!]
\centering
\begin{tikzpicture}
\node [vtx] (a1) at (0,0) {};
\node [vtx] (a2) at (1,0) {};
\node [vtx] (v1) at (2,0) {};
\node [vtx] (v2) at (3,0) {};
\node [vtx] (v3) at (4,0) {};
\node [vtx] (v4) at (5,0) {};
\node [vtx] (v5) at (6,0) {};
\node [vtx] (v6) at (7,0) {};
\node [vtx] (b1) at (8,0) {};
\node [vtx] (b2) at (9,0) {};

\draw [blue] (v1) to node[blabel]{} (v2) to node[blabel]{} (v3) to node[blabel]{} (v4) to node[blabel]{} (v5) to node[blabel]{} (v6);
\draw [red] (a1) to [arch] node [rlabel]{} (v1);
\draw [red] (a2) to [arch] node [rlabel]{} (v2);
\draw [red] (a1) to [arch] node [rlabel]{} (v3);
\draw [red] (v5) to [arch] node [rlabel]{} (b2);
\draw [red] (v6) to node [rlabel]{} (b1);
\draw [decorate,decoration={brace,amplitude=10pt}] (5.2,-0.2) -- (1.8,-0.2);
\node at (3.5,-0.8) {Left-forced};
\draw [decorate,decoration={brace,amplitude=10pt}] (7.2,-0.1) -- (4.8,-0.1);
\node at (6,-0.7) {Right-forced};
\end{tikzpicture}
\caption{An example of left-forced and right-forced edges in the proof of Lemma~\ref{lemma:p3-lower-bound}}
\label{fig:p3-lower-bound}
\end{figure}

Suppose that the game continues until a blue $P_n$ is created. Let $uv, vw$ with $u < v < w$ be two consecutive edges of that path. Then it is impossible for $uv$ to be right-forced and $vw$ to be left-forced; in that case, $v$ would have a red edge both from a preceding vertex and to a following vertex, and a red $P_3$ would already have existed.

Therefore, the blue $P_n$ must consist of a segment (possibly empty) of left-forced edges, followed by a segment (possibly empty) of right-forced edges.

Each left-forced edge is forced by a red edge to its first endpoint from a preceding vertex, and these edges are all different (because their second endpoint is different). Each right-forced edge is forced by a red edge from its second endpoint to a following vertex, and these edges are all different (because their first endpoint is different). The red edges forcing the left-forced edges must be different from the red edges forcing the right-forced edges, because each of the former edges ends to the left of where each of the latter edges begins.

Therefore the $n-1$ blue edges of the blue $P_n$ are forced by $n-1$ distinct red edges, and there must be at least $2(n-1)$ edges total.
\end{proof}

\begin{lemma}\label{lemma:p3-upper-bound}
	For all $n \ge 2$, $r_o(P_3,P_n)\leq \frac{8}{3}n-\frac{10}{3}.$
\end{lemma}

For this proof, we define the \emph{claw strategy} for Builder to be the following. At all times, Builder keeps track of a $k$-vertex blue path, and $\ell$ vertices which are the right endpoints of a red edge: we call these \emph{leftover vertices}. The strategy may start with $k=1$ and $\ell=0$ by picking any vertex to be the single vertex of a blue $P_1$.

On each step of the claw strategy, Builder draws an edge from the rightmost vertex of the blue path to the leftmost vertex Builder is not yet tracking. If it is blue, then the blue path is extended; if it is red, then there is an additional leftover vertex. As a result, after $s$ steps, $k+\ell \ge s+1$. (Conversely, it takes at most $k_0 + \ell_0 - 2$ steps to guarantee that either $k \ge k_0$ or that $\ell \ge \ell_0$.) One possible result of several steps of the claw strategy is shown in Figure~\ref{fig:claw-strategy}.

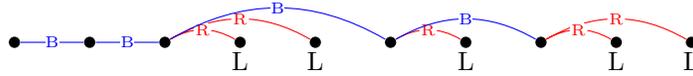
\begin{figure}[h!]
	\centering
	\begin{tikzpicture}
		\foreach \i in {0,1,...,9} \node [vtx] (v\i) at (\i,0) {};
		
		\draw [red] (v2) to [arch] node[rlabel]{} (v3) (v2) to [arch] node[rlabel]{} (v4) (v5) to [arch] node[rlabel]{} (v6) (v7) to [arch] node[rlabel]{} (v8) (v7) to [arch] node[rlabel]{} (v9);
		\draw [blue] (v0) to node[blabel]{} (v1) to node[blabel]{} (v2) to [arch] node[blabel]{} (v5) (v5) to [arch] node[blabel]{} (v7);
		\node [below] at (v3) {\small L};
		\node [below] at (v4) {\small L};
		\node [below] at (v6) {\small L};
		\node [below] at (v8) {\small L};
		\node [below] at (v9) {\small L};
	\end{tikzpicture}

	\caption{One possible outcome of the claw strategy after $9$ steps, resulting in a $5$-vertex blue path and $5$ leftover vertices (marked with an L)}
	\label{fig:claw-strategy}
\end{figure}

The claw strategy may be performed in reverse: extending a path from the left, and obtaining leftover vertices which are the left endpoints of a red edge.

\begin{proof}[Proof of Lemma~\ref{lemma:p3-upper-bound}]
The first phase of Builder's strategy follows the claw strategy for $n-1$ steps, creating a blue path of length $k$ and $n-k$ leftover vertices, where $1 \le k \le n$. The value of $k$ is under Painter's control.

The second phase of Builder's strategy follows the claw strategy in reverse, skipping ahead far enough that all vertices used in the second phase are to the right of all vertices used in the first phase. It lasts for a variable number of steps, depending on $k$:
\begin{itemize}
\item When $k \ge \frac{n+1}{3}$, the second phase lasts $n-2$ steps, creating either a blue $k$-vertex path or $n-k$ leftover vertices.  
\item When $k \le \frac n3$, the second phase lasts $2k-2$ steps, creating either a blue $k$-vertex path or $k$ leftover vertices.
\end{itemize}
The final phase of Builder's strategy depends on the outcome of the second phase.

\textbf{Case 1: $k \ge \frac{n+1}{3}$ and a blue $k$-vertex path is created in the second phase.} In this case, Builder's last $n-k$ moves extend this path through all the leftover vertices of the first phase. If Painter colors any of these edges red, a red $P_3$ is formed; otherwise, a blue $P_n$ is formed. The total number of steps is $(n-1)+(n-2)+(n-k) = 3n-k-3 \le \frac{8n-10}{3}$.

\textbf{Case 2: $k \le \frac n3$ and a blue $k$-vertex path is created in the second phase.} Builder's follow-up is the same as in case 1. The total number of steps is $(n-1)+(2k-2)+(n-k) = 2n+k-3 \le \frac{7n-9}{3}$.

\textbf{Case 3: $k \ge \frac{n+1}{3}$ and $n-k$ leftover vertices are created in the second phase.} In this case, Builder's last $n-k$ moves extend the blue $k$-vertex path from the first phase through the leftover vertices of the second phase. If Painter colors any of these edges red, a red $P_3$ is formed; otherwise, a blue $P_n$ is formed. The total number of steps is $(n-1)+(n-2)+(n-k) = 3n-k-3 \le \frac{8n-10}{3}$.

\textbf{Case 4: $k \le \frac n3$ and $k$ leftover vertices are created in the second phase.} In this case, Builder's last $n-1$ moves build a path through the leftover vertices of both phases: $n$ vertices total. If any of these edges are red, a red $P_3$ is formed; otherwise, a blue $P_n$ is formed. The total number of steps is $(n-1)+(2k-2)+(n-1) \le \frac{8n-12}{3}$.

In all cases, Builder wins in at most $\frac83n - \frac{10}{3}$ steps.
\end{proof}

\subsection{The second bad graph: $X$}

\begin{lemma}\label{lemma:XPn} We have \[\frac{3(n-1)}{2}\leq r_o(X,P_n)\leq\frac32n+2.\]
\end{lemma}

\begin{proof}
To prove the lower bound, assume Painter follows the strategy of coloring each edge red unless a red copy of $X$ is created. Builder wins the game when Painter creates a blue copy of $P_n$. Each edge of this blue path must intersect a red edge, and each red edge can intersect at most two of the blue path edges. Therefore, there must be a total of at least $n-1$ blue edges and $\frac{n-1}{2}$ red edges when Builder wins, giving $r_o(X,P_n)\geq  \frac{3(n-1)}{2}$.

For the upper bound, we present a strategy for Builder. For convenience, we define a special step in Builder's strategy called \emph{lacing}. Suppose that the graph built so far contains vertices
\[
	a_1 < \dots < a_k < b_1 < c_1 < d_1 < \dots < d_\ell
\]
with a blue copy of $P_k$ through $a_1, \dots, a_k$, a red edge $b_1c_1$, and a blue copy of $P_\ell$ through $d_1, \dots, d_\ell$. To \emph{lace the blue paths together through} $b_1c_1$, Builder begins by playing edges $b_2c_2$, $b_3c_3$, and so on with $b_{i-1} < b_i < c_i < c_{i-1}$. Builder stops after playing $b_mc_m$ if either $b_mc_m$ is blue, or if $k + (2m-1) + \ell \ge n$.

If $b_mc_m$ is blue, Builder plays the edges of an ordered path from $a_k$ to $b_m$ through $m-2$ additional vertices interleaving $b_1, \dots, b_{m-1}$, and an ordered path from $c_m$ to $d_1$ through $m-2$ additional vertices interleaving $c_{m-1}, \dots, c_1$. If any of these edges are red, a red copy of $X$ is created; otherwise, a blue copy of $P_{k+2m-2+\ell}$ is created. In this case, lacing the paths together took $3m-3$ moves.

If $b_mc_m$ is red, Builder instead plays the edges of an ordered path from $a_k$ to $d_1$ through $2m-1$ additional vertices interleaving $b_1, \dots, b_{m},c_m, \dots c_1$. If any of these edges are red, a red copy of $X$ is created; otherwise, a blue copy of $P_{k+(2m-1)+\ell}$ is created. In either case, Builder wins; in this case, lacing the paths together took $3m-1$ moves.

See Figure~\ref{fig:lacing-strategy} for an example of Builder lacing together two blue paths $P_3$ and $P_4$ using a red edge.

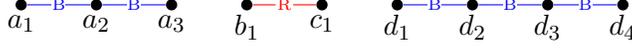
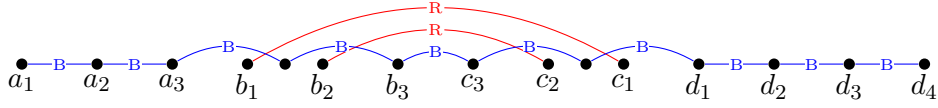
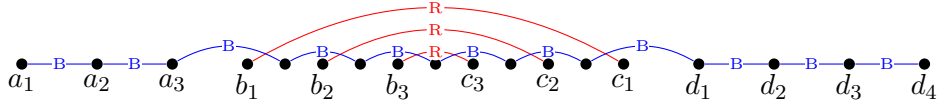
\begin{figure}
\centering
\begin{subfigure}{\textwidth}
\centering
\begin{tikzpicture}
	\node [vtx] (a1) at (1,0) {}; \node [below] at (a1) {$a_1$};
	\node [vtx] (a2) at (2,0) {}; \node [below] at (a2) {$a_2$};
	\node [vtx] (a3) at (3,0) {}; \node [below] at (a3) {$a_3$};
	
    \node [vtx] (b1) at (4,0) {}; \node [below] at (b1) {$b_1$};
	\node [vtx] (c1) at (5,0) {}; \node [below] at (c1) {$c_1$};

	\node [vtx] (d1) at (6,0) {}; \node [below] at (d1) {$d_1$};
	\node [vtx] (d2) at (7,0) {}; \node [below] at (d2) {$d_2$};
	\node [vtx] (d3) at (8,0) {}; \node [below] at (d3) {$d_3$};
	\node [vtx] (d4) at (9,0) {}; \node [below] at (d4) {$d_4$};

	\draw [blue] (a1) to node[blabel]{} (a2) to node[blabel]{} (a3);
    \draw [blue] (d1) to node[blabel]{} (d2) to node[blabel]{} (d3) to node[blabel]{} (d4);
    \draw [red] (b1) to node [rlabel]{} (c1);
 

\end{tikzpicture}
\caption{State before Builder laces the blue paths $a_1, a_2, a_3$ and $d_1,d_2,d_3,d_4$ together through the red edge $b_1c_1$}
\label{fig:lacing-setup}
\end{subfigure}

\begin{subfigure}{\textwidth}
\centering
\begin{tikzpicture}
	\node [vtx] (a1) at (1,0) {}; \node [below] at (a1) {$a_1$};
	\node [vtx] (a2) at (2,0) {}; \node [below] at (a2) {$a_2$};
	\node [vtx] (a3) at (3,0) {}; \node [below] at (a3) {$a_3$};
	
    \node [vtx] (b1) at (4,0) {}; \node [below] at (b1) {$b_1$};
    \node [vtx] (b2) at (5,0) {}; \node [below] at (b2) {$b_2$};
    \node [vtx] (b3) at (6,0) {}; \node [below] at (b3) {$b_3$};
    \node [vtx] (c3) at (7,0) {}; \node [below] at (c3) {$c_3$};
    \node [vtx] (c2) at (8,0) {}; \node [below] at (c2) {$c_2$};
	\node [vtx] (c1) at (9,0) {}; \node [below] at (c1) {$c_1$};

	\node [vtx] (d1) at (10,0) {}; \node [below] at (d1) {$d_1$};
	\node [vtx] (d2) at (11,0) {}; \node [below] at (d2) {$d_2$};
	\node [vtx] (d3) at (12,0) {}; \node [below] at (d3) {$d_3$};
	\node [vtx] (d4) at (13,0) {}; \node [below] at (d4) {$d_4$};

 	\node [vtx] (x1) at (4.5,0) {}; 
 	\node [vtx] (x2) at (8.5,0) {}; 

	\draw [blue] (a1) to node[blabel]{} (a2) to node[blabel]{} (a3);
    \draw [blue] (d1) to node[blabel]{} (d2) to node[blabel]{} (d3) to node[blabel]{} (d4);
    \draw [red] (b1) to [arch] node [rlabel]{} (c1);
    \draw [red] (b2) to [arch] node [rlabel]{} (c2);
    \draw [blue] (b3) to [arch] node [blabel]{} (c3);
 
	\draw [blue] (a3) to [arch] node [blabel]{} (x1)  to [arch] node [blabel]{} (b3);
	\draw [blue] (c3) to [arch] node [blabel]{} (x2) to [arch] node [blabel]{} (d1);

\end{tikzpicture}
\caption{One possibility for how Builder laces the paths together to create a blue $P_{11}$}
\label{fig:lacing-outcome1}
\end{subfigure}

\begin{subfigure}{\textwidth}
\centering
\begin{tikzpicture}
	\node [vtx] (a1) at (1,0) {}; \node [below] at (a1) {$a_1$};
	\node [vtx] (a2) at (2,0) {}; \node [below] at (a2) {$a_2$};
	\node [vtx] (a3) at (3,0) {}; \node [below] at (a3) {$a_3$};
	
    \node [vtx] (b1) at (4,0) {}; \node [below] at (b1) {$b_1$};
    \node [vtx] (b2) at (5,0) {}; \node [below] at (b2) {$b_2$};
    \node [vtx] (b3) at (6,0) {}; \node [below] at (b3) {$b_3$};
    \node [vtx] (c3) at (7,0) {}; \node [below] at (c3) {$c_3$};
    \node [vtx] (c2) at (8,0) {}; \node [below] at (c2) {$c_2$};
	\node [vtx] (c1) at (9,0) {}; \node [below] at (c1) {$c_1$};

	\node [vtx] (d1) at (10,0) {}; \node [below] at (d1) {$d_1$};
	\node [vtx] (d2) at (11,0) {}; \node [below] at (d2) {$d_2$};
	\node [vtx] (d3) at (12,0) {}; \node [below] at (d3) {$d_3$};
	\node [vtx] (d4) at (13,0) {}; \node [below] at (d4) {$d_4$};

 	\node [vtx] (x1) at (4.5,0) {};
    \node [vtx] (x2) at (5.5,0) {};
    \node [vtx] (x3) at (6.5,0) {};
 	\node [vtx] (x4) at (7.5,0) {}; 
 	\node [vtx] (x5) at (8.5,0) {}; 

	\draw [blue] (a1) to node[blabel]{} (a2) to node[blabel]{} (a3);
    \draw [blue] (d1) to node[blabel]{} (d2) to node[blabel]{} (d3) to node[blabel]{} (d4);
    \draw [red] (b1) to [arch] node [rlabel]{} (c1);
    \draw [red] (b2) to [arch] node [rlabel]{} (c2);
    \draw [red] (b3) to [arch] node [rlabel]{} (c3);
 
	\draw [blue] (a3) to [arch] node [blabel]{} (x1) to [arch] node [blabel]{} (x2) to [arch] node [blabel]{} (x3) to [arch] node [blabel]{} (x4) to [arch] node [blabel]{} (x5) to [arch] node [blabel]{} (d1);

\end{tikzpicture}
\caption{The other possibility for how Builder laces the paths together to create a blue $P_{12}$}
\label{fig:lacing-outcome2}
\end{subfigure}

\caption{An example of the lacing strategy Builder uses in the proof of Lemma~\ref{lemma:XPn}}
\label{fig:lacing-strategy}
\end{figure}

Builder's overall strategy maintains a blue path $v_1, \dots, v_k$, which is extended by lacing; before any moves have been made, Builder can take $k=1$ by choosing an arbitrary vertex $v_1$. Builder extends this path in two phases.
 
In the first phase, fix $n$ additional vertices $w_1, \dots, w_n$ appearing after $v_k$ in reverse order: $v_k < w_n < \dots < w_2 < w_1$. Builder plays the following edges of the ordered path through these $n$ vertices:

\begin{enumerate}
\item Edge $w_3 w_2$. If this edge is red, the first phase ends (leading to case 1 below).
\item Edge $w_4 w_3$. If this edge is red, Builder plays $w_2 w_1$ and then, regardless of its color, the first phase ends (leading to case 2 or case 3 below).
\item Edges $w_5 w_4$, $w_6 w_5$, and so on, ending the first phase either once one of these edges is red (also leading to case 2 below) or once a blue path on at least $n-k+1$ vertices is created (leading to case 4 below).
\end{enumerate}

Builder's strategy in the second phase depends on the result of the first phase.

\textbf{Case 1.} $w_3 w_2$ is red. In this case, choose vertices $v_{k+1}$ and $v_{k+2}$ such that $w_3 < v_{k+1} < w_2 < v_{k+2}$; Builder plays edges $v_k v_{k+1}$ and $v_{k+1}v_{k+2}$. If either edge is red, a red copy of $X$ is created; otherwise, Builder extends the blue path to $k+2$ vertices.

\textbf{Case 2.} The edges played in the first phase form an ordered path on $\ell+2$ vertices, for $2 \le \ell \le n-k-1$, in which the first edge is red and all other edges are blue. In this case, Builder laces the blue path on $v_1, \dots, v_k$ to the last $\ell$ vertices of this path through the red edge.

\textbf{Case 3.} $w_4 w_3$ and $w_2 w_1$ are both red. In this case, choose an arbitrary vertex $x$ such that $w_3 < x < w_2$, and an arbitrary vertex $y$ such that $w_1 < y$. Builder laces the blue path on $v_1, \dots, v_k$ to $x$ through $w_4 w_3$. If this blue path still has fewer than $n$ vertices, Builder laces it to $y$ through $w_2w_1$.

\textbf{Case 4.} The edges played in the first phase form an ordered path on $t \ge n-k+1$ vertices $w_{t+1}, \dots, w_2$. In this case, Builder plays edges $v_{k-1} w_{t+1}$ and $v_k w_t$. If both of these edges are red, a red copy of $X$ is created, and Builder wins. Otherwise, a blue path through all but one of the vertices $v_1, \dots, v_k$ and $w_{t+1}, \dots, w_2$ is created, with at least $k+t-1 \ge n$ vertices, so Builder still wins.

Now we analyze the number of moves required for Builder to win using this strategy. To begin with, while the blue path $v_1, \dots, v_k$ still has fewer than $n$ vertices, at most $\frac32(k-1)$ edges have been played, which we prove by induction. Only $0 = \frac32(1-1)$ moves are required to reach $k-1$. The number of vertices added to the path varies by case. If we ignore the lacing steps (which add a variable number of vertices), then:
\begin{itemize}
    \item In case 1, $2$ vertices are added, and $3 = \frac32(2)$ edges are played.
    \item In case 2, $\ell$ vertices are added, and $\ell+1 \le \frac32 \ell$ edges are played.
    \item In case 3, $2$ vertices ($x$ and $y$) are added, and $3 = \frac32(2)$ edges are played.
    \item Case 4 can be ignored for now, since it always ends in a win for Builder.
\end{itemize}
Additionally, each lacing step that does not result in a win for Builder adds $2m-2$ vertices to the blue path in $3m-3 = \frac32(2m-2)$ moves. Altogether, we confirm that every time Builder extends the path to a length $k$ but does not win, at most $\frac32(k-1)$ moves have been made.

A constant term is added at the end, when the path reaches or exceeds $n$ steps. A general reason for this is that in such a case, lacing two paths together may add $2m-1$ vertices to the path at the cost of $3m-1 = \frac32(2m-1) + \frac12$ moves: adding a one-time cost of $\frac12$. The constant term also depends on the case:
\begin{itemize}
    \item In case 1, the efficiency does not change, but the final path might have up to $n+1$ vertices, since $2$ vertices are added at once, at most $\frac32n$ moves have been made.
    \item In case 2, when lacing the two paths together adds $2m-1$ vertices to the length in $3m-1$ moves, this may also result in a path on $n+1$ vertices; together with the one-time cost of $\frac12$, at most $\frac32n + \frac12$ moves have been made.
    \item In case 3, if Builder wins after lacing $y$ to the path, the analysis is the same as for case 2. However, Builder may also win after lacing $x$ to the path; in that case, the $3$ edges played before lacing only contributed one vertex ($x$) for an additional cost of $\frac32$; st most $\frac32n + 2$ moves have been made.
    \item In case 4, Builder adds $t-1$ vertices to the path in $t+1$ moves; at minimum, $t\ge 3$, since the added vertices include $w_4, w_3, w_2$, so $t+1 \le \frac32t - \frac12$. Therefore the final path of length $k+t-1$ is created in at most $\frac32(k-1+t) - \frac12$ moves. If $t = n-k+1$, this is $\frac32n - \frac12$ moves, and in most cases, Builder stops as soon as $t$ reaches $n-k+1$. However, $t$ is always at least $3$, even if $k = n-1$, so Builder might instead create a path of length $n+1$ in $\frac32(n+1) - \frac12$ or $\frac32n + 1$ moves.
\end{itemize}
All in all, we can guarantee that when Builder wins, at most $\frac32n + 2$ moves have been made.
\end{proof}

\subsection{The third bad graph: $K_{1,2}$}

The graph $K_{1,2}$ is unique among the bad graphs of Theorem~\ref{thm:bad-small-graphs} in that we can determine the exact value of $r_o(K_{1,2}, P_n)$. In fact, we can find $r_o(K_{1,k}, P_n)$ exactly for all $k$.

\begin{lemma}\label{lemma:k1k}
For all $n, k \ge 1$, $r_o(K_{1,k}, P_n) = k(n-1)$.
\end{lemma}
\begin{proof}
For the lower bound, Painter's strategy is to color every edge red, except if a red copy of $K_{1,k}$ is created. Builder wins the game when a blue copy of $P_n$ is created. Let $vw$ be an edge of this path, with $v<w$; in order for Painter to color $vw$ blue, there must have been a red copy of $K_{1,k-1}$ whose leftmost vertex is $v$. In total, there are $n-1$ red copies of $K_{1,k-1}$, all with distinct edges. (Figure~\ref{fig:k14-strategy-painter} shows an example with $n=4$ and $k=4$.) Together with the $n-1$ blue edges of $P_n$, there are at least $k(n-1)$ edges total.

For the upper bound, Builder will successively build a blue path $P$; initially, $P$ consists of a single arbitrary vertex. When $P = v_1 v_2 \dots v_s$ for some $s \ge 1$, Builder chooses vertices $w_1, \dots, w_k$ to the right of $v_s$ and plays the $k$ edges $\{v_sw_1, v_s w_2, \dots, v_s w_k\}$. If all $k$ edges are red, a red copy of $K_{1,k}$ is created and Builder wins. If $v_sw_i$ is blue for some $i$, then $P$ can be extended to $v_1 v_2 \dots v_s w_i$. (Figure~\ref{fig:k14-strategy-builder} shows an example of the second possibility, with $s=5$ and $k=4$.) Every $k$ moves, Builder is able to extend $P$ by one more vertex; thus Builder wins in $k(n-1)$ moves.
\end{proof}

\begin{figure}[h!]
\centering
\begin{subfigure}{\textwidth}
\centering
\begin{tikzpicture}
	\node [vtx] (v1) at (1,0) {}; \node [below] at (v1) {$v_1$};
	\node [vtx] (a1) at (2,0) {};
	\node [vtx] (b1) at (3,0) {};
	\node [vtx] (c1) at (3,0) {};
	\node [vtx] (v2) at (4,0) {}; \node [below] at (v2) {$v_2$};
	\node [vtx] (a2) at (5,0) {};
	\node [vtx] (b2) at (6,0) {};
	\node [vtx] (v3) at (7,0) {}; \node [below] at (v3) {$v_3$};
	\node [vtx] (a3) at (8,0) {};
	\node [vtx] (b3) at (9,0) {};	
	\node [vtx] (c2) at (10,0) {};
	\node [vtx] (v4) at (11,0) {}; \node [below] at (v4) {$v_4$};
	\node [vtx] (c3) at (12,0) {};

	\draw [blue] (v1) to [arch] node [blabel]{} (v2) to [arch] node [blabel]{} (v3) to [arch] node [blabel]{} (v4);
	\draw [red] (v1) to node [rlabel]{} (a1);
	\draw [red] (v1) to [arch] node [rlabel]{} (b1);
	\draw [red] (v1) to [arch] node [rlabel]{} (c1);
	\draw [red] (v2) to node [rlabel]{} (a2);
	\draw [red] (v2) to [arch] node [rlabel]{} (b2);
	\draw [red] (v2) to [arch] node [rlabel]{} (c2);
	\draw [red] (v3) to node [rlabel]{} (a3);
	\draw [red] (v3) to [arch] node [rlabel]{} (b3);
	\draw [red] (v3) to [arch] node [rlabel]{} (c3);
\end{tikzpicture}
\caption{An example of the structure that must exist to force Painter to color the path $v_1v_2v_3v_4$ blue}
\label{fig:k14-strategy-painter}
\end{subfigure}


\begin{subfigure}{\textwidth}
\centering
\begin{tikzpicture}
	\node [vtx] (v1) at (1,0) {}; \node [below] at (v1) {$v_1$};
	\node [vtx] (v2) at (2,0) {}; \node [below] at (v2) {$v_2$};
	\node [vtx] (v3) at (3,0) {}; \node [below] at (v3) {$v_3$};
	\node [vtx] (v4) at (4,0) {}; \node [below] at (v4) {$v_4$};
	\node [vtx] (v5) at (5,0) {}; \node [below] at (v5) {$v_5$};

	\node [vtx] (w1) at (6,0) {}; \node [below] at (w1) {$w_1$};
	\node [vtx] (w2) at (7,0) {}; \node [below] at (w2) {$w_2$};
	\node [vtx] (w3) at (8,0) {}; \node [below] at (w3) {$w_3$};
	\node [vtx] (w4) at (9,0) {}; \node [below] at (w4) {$w_4$};

	\draw [blue] (v1) to node[blabel]{} (v2) to node[blabel]{} (v3) to node[blabel]{} (v4) to node[blabel]{} (v5);

	\draw [red] (v5) to [arch] node [rlabel]{} (w1);
	\draw [red] (v5) to [arch] node [rlabel]{} (w2);
	\draw [blue] (v5) to [arch] node [blabel]{} (w3);
	\draw [red] (v5) to [arch] node [rlabel]{} (w4);

\end{tikzpicture}
\caption{Builder extends a blue $P_5$ to a blue $P_6$ via edge $v_5w_3$}
\label{fig:k14-strategy-builder}
\end{subfigure}

\caption{Diagrams for the proof of Lemma~\ref{lemma:k1k}, in the case $k=4$}
\label{fig:k14-strategy}
\end{figure}
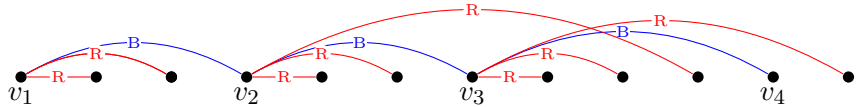
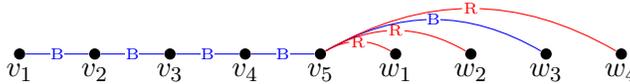

\section{Results for matchings}
\label{section:n+O(1)}

\subsection{Serial matchings}

In this section, we analyze the online Ramsey number $r_o(M_k, P_n)$, where $M_k$ is the serial $k$-edge matching. This case is notable for two reasons. First, $M_k$ is the instance we see of graphs $G$ such that $r_o(G,P_n) = n + O(1)$ as $n \to \infty$. Second, we will be able to find $r_o(M_k,P_n)$ exactly, when $n$ is sufficiently large compared to $k$. 

We will prove Theorem~\ref{thm:serial-matching} in two steps. We begin with the upper bound on $r_o(M_k,P_n)$, which is a strategy for Builder.

\begin{lemma}\label{lemma:mk-upper-bound}
If $n \ge 4$ and $k \ge 2$, then $r_o(M_k, P_n) \le n + 2k -4$.
\end{lemma}
\begin{proof}
In this proof, we assume that Builder takes care to space out vertices so that it is always possible to play a vertex between any two existing vertices.

To keep track of Builder's progress, we say that the graph built so far is in state $(x,y)$ if it contains a blue $P_x$ followed by a red $M_y$, and we assign state $(x,y)$ the \emph{weight} $x+2y$. (An example is shown in Figure~\ref{fig:serial-strategy-state}.) The empty graph is in state $(1,0)$ with weight $1$, because a blue $P_1$ is just a vertex. A graph with no red $M_k$ or blue $P_n$ must be in a state with weight at most $n-1 + 2(k-1) = n+2k-3$.

\begin{figure}
\centering
\begin{subfigure}{\textwidth}
\centering
\begin{tikzpicture}
	\node [vtx] (v1) at (1,0) {}; \node [below] at (v1) {$v_1$};
	\node [vtx] (v2) at (2,0) {}; \node [below] at (v2) {$v_2$};
	\node [vtx] (v3) at (3,0) {}; \node [below] at (v3) {$v_3$};
	\node [vtx] (v4) at (4,0) {}; \node [below] at (v4) {$v_4$};
	\node [vtx] (v5) at (5,0) {}; \node [below] at (v5) {$v_5$};

	\node [vtx] (x1) at (9,0) {};
	\node [vtx] (y1) at (10,0) {};
	\node [vtx] (x2) at (11,0) {};
	\node [vtx] (y2) at (12,0) {};
	\node [vtx] (x3) at (13,0) {};
	\node [vtx] (y3) at (14,0) {};

	\draw [blue] (v1) to node[blabel]{} (v2);
	\draw [blue] (v2) to node[blabel]{} (v3);
	\draw [blue] (v3) to node[blabel]{} (v4);
	\draw [blue] (v4) to node[blabel]{} (v5);

	\draw [red] (x1) to node[rlabel]{} (y1);
	\draw [red] (x2) to node[rlabel]{} (y2);
	\draw [red] (x3) to node[rlabel]{} (y3);
\end{tikzpicture}
\caption{State $(5,3)$ with weight $11$}
\label{fig:serial-strategy-state}
\end{subfigure}	

\bigskip

\begin{subfigure}{\textwidth}
\centering
\begin{tikzpicture}
	\node [vtx] (v1) at (1,0) {}; \node [below] at (v1) {$v_1$};
	\node [vtx] (v2) at (2,0) {}; \node [below] at (v2) {$v_2$};
	\node [vtx] (v3) at (3,0) {}; \node [below] at (v3) {$v_3$};
	\node [vtx] (v4) at (4,0) {}; \node [below] at (v4) {$v_4$};
	\node [vtx] (v5) at (5,0) {}; \node [below] at (v5) {$v_5$};
	
	\node [vtx] (w) at (7,0) {}; \node [below] at (w) {$w$};

	\node [vtx] (x1) at (9,0) {};
	\node [vtx] (y1) at (10,0) {};
	\node [vtx] (x2) at (11,0) {};
	\node [vtx] (y2) at (12,0) {};
	\node [vtx] (x3) at (13,0) {};
	\node [vtx] (y3) at (14,0) {};

	\draw [blue] (v1) to node[blabel]{} (v2);
	\draw [blue] (v2) to node[blabel]{} (v3);
	\draw [blue] (v3) to node[blabel]{} (v4);
	\draw [blue] (v4) to node[blabel]{} (v5);

	\draw [red] (x1) to node[rlabel]{} (y1);
	\draw [red] (x2) to node[rlabel]{} (y2);
	\draw [red] (x3) to node[rlabel]{} (y3);
	
	\draw [blue] (v5) to [arch] node [blabel]{} (w);

\end{tikzpicture}
\caption{Edge $v_5w$ is blue, resulting in state $(6,3)$ with weight $12$}
\label{fig:serial-strategy-blue}
\end{subfigure}

\bigskip

\begin{subfigure}{\textwidth}
\centering
\begin{tikzpicture}
	\node [vtx] (v1) at (1,0) {}; \node [below] at (v1) {$v_1$};
	\node [vtx] (v2) at (2,0) {}; \node [below] at (v2) {$v_2$};
	\node [vtx] (v3) at (3,0) {}; \node [below] at (v3) {$v_3$};
	\node [vtx] (v4) at (4,0) {}; \node [below] at (v4) {$v_4$};
	\node [vtx] (v5) at (5,0) {}; \node [below] at (v5) {$v_5$};
	
	\node [vtx] (w) at (7,0) {}; \node [below] at (w) {$w$};

	\node [vtx] (x1) at (9,0) {};
	\node [vtx] (y1) at (10,0) {};
	\node [vtx] (x2) at (11,0) {};
	\node [vtx] (y2) at (12,0) {};
	\node [vtx] (x3) at (13,0) {};
	\node [vtx] (y3) at (14,0) {};

	\draw [blue] (v1) to node[blabel]{} (v2);
	\draw [blue] (v2) to node[blabel]{} (v3);
	\draw [blue] (v3) to node[blabel]{} (v4);
    \draw [dashed] (v4) to (v5);

	\draw [red] (x1) to node[rlabel]{} (y1);
	\draw [red] (x2) to node[rlabel]{} (y2);
	\draw [red] (x3) to node[rlabel]{} (y3);

	\draw [red] (v5) to [arch] node [rlabel]{} (w);

\end{tikzpicture}
\caption{Edge $v_5w$ is red, resulting in state $(4,4)$ with weight $12$}
\label{fig:serial-strategy-red}
\end{subfigure}

\bigskip

\begin{subfigure}{\textwidth}
\centering
\begin{tikzpicture}
	\node [vtx] (v0) at (0,0) {}; \node [below] at (v0) {$v_0$};
	\node [vtx] (v1) at (1,0) {}; \node [below] at (v1) {$v_1$};
	\node [vtx] (v2) at (2,0) {}; \node [below] at (v2) {$v_2$};
	\node [vtx] (v3) at (3,0) {}; \node [below] at (v3) {$v_3$};
	\node [vtx] (v4) at (4,0) {}; \node [below] at (v4) {$v_4$};
	\node [vtx] (v5) at (5,0) {}; \node [below] at (v5) {$v_5$};

	\node [vtx] (x1) at (9,0) {};
	\node [vtx] (y1) at (10,0) {};
	\node [vtx] (x2) at (11,0) {};
	\node [vtx] (y2) at (12,0) {};

	\draw [red] (v0) to node[rlabel]{} (v1);
	\draw [blue] (v1) to node[blabel]{} (v2);
	\draw [blue] (v2) to node[blabel]{} (v3);
	\draw [blue] (v3) to node[blabel]{} (v4);
	\draw [blue] (v4) to node[blabel]{} (v5);

	\draw [red] (x1) to node[rlabel]{} (y1);
	\draw [red] (x2) to node[rlabel]{} (y2);

\end{tikzpicture}
\caption{State $(1,5,2)$, also with weight $11$}
\label{fig:serial-strategy-refined}
\end{subfigure}
\caption{Diagrams for the proof of Lemma~\ref{lemma:mk-upper-bound}}
\label{fig:serial-strategy}
\end{figure}
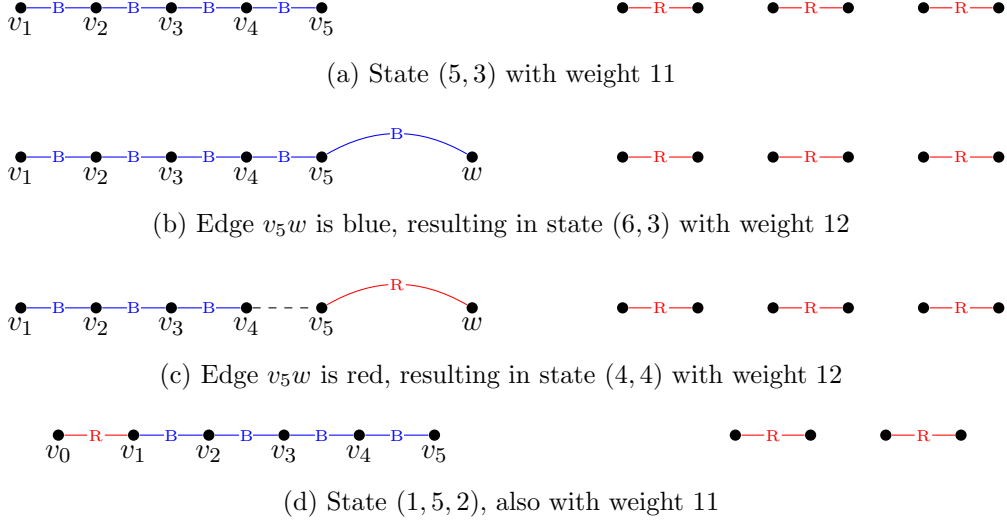

In state $(x,y)$, let $v_1, v_2, \dots, v_x$ be the vertices of the blue $P_x$ in order, and let $w$ be a vertex to the right of $v_x$ and to the left of the red $M_y$. Suppose Builder plays the edge $v_x w$. Then:
\begin{itemize}
\item If $v_x w$ is blue, the graph now contains a blue $P_{x+1}$ (the path $v_1 v_2 \dots v_x w$) followed by a red $M_y$, so it is in state $(x+1,y)$. (An example is shown in Figure~\ref{fig:serial-strategy-blue}.)
\item If $v_x w$ is red, the graph now contains a blue $P_{x-1}$ (the path $v_1 v_2 \dots v_{x-1}$) followed by a red $M_{y+1}$ (edge $v_x w$ together with the previous $M_y$), so it is in state $(x-1,y+1)$. (An example is shown in Figure~\ref{fig:serial-strategy-red}; the dashed edge is the edge $v_{x-1}v_x$, which is no longer part of the blue path tracked by Builder.)
\end{itemize}
In both cases, the new state of the graph has weight $x+2y+1$: it increases by $1$. This argument alone proves $r_o(M_k, P_n) \le n + 2k-3$: after $n+2k-3$ repetitions of the move above from state $(1,0)$, Builder reaches a state $(x,y)$ with weight $n+2k-2$, so either $x \ge n$ or $y \ge k$. 

To improve this bound by $1$, we refine the strategy. We say that the graph is in state $(1,x,y)$ if there is a path $P_{1+x}$ whose first edge is red and all other edges are blue, followed by a red $M_y$; we assign this state weight $x+2y+2$. (An example is shown in Figure~\ref{fig:serial-strategy-refined}.) It remains the case that when the graph is in a state $(1,x,y)$ with weight at least $n+2k-2$, it contains a red $M_k$ or a blue $P_n$.

From state $(1,x,y)$ with $x>1$, Builder can follow the strategy above to obtain state $(1,x+1,y)$ or $(1,x-1,y+1)$ in one move, increasing the weight by $1$. State $(1,1,y)$ carries the same information as state $(1,y+1)$ and has the same weight, so if the graph reaches such a state, we switch to treating it as state $(1,y+1)$.

Builder begins with an ``opening," laying down a path starting at a vertex and going left, until the first time Painter colors an edge of this path red. Either Builder wins in $n-1 < n+2k-4$ steps, or else achieves state $(1,x,0)$ in $x$ steps, which has weight $x+2$. From here, Builder can follow the strategy above to increase weight by at least $1$ with every step. In $n+2k-4$ steps, a state with weight $n+2k-2$ is reached, guaranteeing that Builder wins.
\end{proof}

For sufficiently large $n$, Painter has a simple counter-strategy: Painter colors every edge blue unless this would create a blue $P_n$. In the lemma that follows, we prove that this strategy avoids losing in fewer than $n+2k-4$ moves when $n \ge \max\{3k-3,2k+1 \}$ and $k\ge 2$, completing the proof of Theorem~\ref{thm:serial-matching}. 

\begin{lemma}\label{lemma:mk-lower-bound}
	If $k\ge 2$ and $n \ge \max\{3k-3,2k+1 \}$, then $r_o(M_k, P_n) \ge n + 2k-4$.
\end{lemma}
\begin{proof}
When Painter colors every edge blue unless this would create a blue $P_n$, the end state of the game is an ordered graph $G$ with $k$ red edges forming a red $M_k$. (We may assume that Builder does not play any other edges that Painter would color red.) We will show that $G$ must contain at least $n + 2k-4$ edges.

As a consequence of Painter's strategy, $G$ does not contain a blue $P_n$. Moreover, every red edge $e \in E(G)$ must be part of a $P_n$ in $G$ in which $e$ is the only red edge. (Otherwise, Painter would have colored $e$ blue.) It follows that if $vw$ is a red edge, then there is no blue $v-w$ path in $G$: otherwise, replacing $vw$ by the blue $v-w$ path would create a path of length at least $n$ in $G$ with no red edges.

Define a path $P^*$ as follows. 
\begin{enumerate}
\item If there is a $P_{n+1}$ in $G$ in which the first and last edges are red, and all others are blue, let $P^*$ be this $P_{n+1}$.
\item Otherwise, if there is a $P_n$ in $G$ in which the first or last edge is red, and all others are blue, let $P^*$ be this $P_n$. (Note that in this case, the other end of $P^*$ is not incident to a red edge, or else we would be in the first case.)
\item Otherwise, choose any edge $e$, and let $P^*$ be a $P_n$ in $G$ in which $e$ is the only red edge.
\end{enumerate}
In all cases, $P^*$ contains $n-2$ blue edges and either $1$ or $2$ red edges.

We call every vertex of $P^*$ that is not incident on a red edge of $P^*$ an \emph{anchor}. Furthermore, we assign a direction (``left" or ``right") to each anchor by the following algorithm:
\begin{enumerate}
\item If an anchor is the rightmost vertex of a blue $P_{n-k}$ in $G$, label it a \emph{left anchor}.
\item If an anchor is the leftmost vertex of a blue $P_{n-k}$ in $G$, label it a \emph{right anchor}. No anchor will fall in both categories: this would imply the existence of a blue $P_{2n-2k-1}$, which contains a blue $P_n$.
\item At this point, if there is a left anchor to the left of a right anchor, the corresponding paths are disjoint and have $2(n-k-1)$ blue edges; together with the $k$ red edges, we get $2n-k-1$, which is at least $n+2k-4$ when $n \ge 3k-3$.

Assign directions to the other anchors arbitrarily so that this rule continues to hold: no left anchor is to the left of a right anchor.
\end{enumerate}

\begin{claim}\label{claim:anchors}
Let $vw$ be a red edge which is not part of $P^*$, but which has an endpoint on $P^*.$ Then:
\begin{enumerate}
    \item[(i)] Exactly one of $v$ or $w$ is an anchor.
    \item[(ii)] If $v$ is an anchor, then $w$ is the left endpoint of a blue edge.
    \item[(iii)] Symmetrically, if $w$ is an anchor, then $v$ is the right endpoint of a blue edge.
\end{enumerate}
\end{claim}
\begin{proof}
Suppose for contradiction that $v$ and $w$ are both anchors. If they are on the same blue segment of $P^*$, then there is a blue $v-w$ path, which we ruled out earlier. If they are on different blue segments of $P^*$, then the red edge of $P^*$ between them is nested between the endpoints of the red edge $vw$, which violates our assumption that the red edges form an $M_k$. This proves (i).

We only prove (ii), since (iii) is its mirror image. Suppose $v$ is an anchor. We know that there is a copy of $P_n$ in $G$ in which $vw$ is the only red edge. That copy of $P_n$ includes a blue edge satisfying (ii), unless $vw$ is its last edge. In this case, $v$ is the rightmost endpoint of a blue $P_{n-1}$. This shows that $G$ contains a $P_n$ in which the last edge is red, and all others are blue; therefore the path $P^*$ is also chosen to contain a blue $P_{n-1}$, extended on one or both sides by a red edge.

If $v$ is rightmost vertex of $P^*$, we would have included edge $vw$ as part of $P^*$; a contradiction. But in all other cases, $v$ is the left endpoint of a blue edge of $P^*$, and we obtain a blue $P_n$; another contradiction. Therefore $v$ cannot be the rightmost endpoint of a blue $P_{n-1}$, and (ii) follows.
\end{proof}

To prove the lemma, it suffices to show that that in $G - P^*$, there are at least as many blue edges as red edges. When $P^*$ has $n-2$ blue edges and $1$ red edge, this would give us $k-1$ red edges and at least $k-1$ blue edges in $G-P^*$, for a total of at least $n+2k-3$. When $P^*$ has $n-2$ blue edges and $2$ red edges, this would give us  $k-2$ red edges and at least $k-2$ blue edges in $G-P^*$, for a total of at least $n+2k-4$. To prove this, we will consider the connected components of $G - P^*$, and show that each one has at least as many blue edges as red edges.

Let $C$ be a connected component of $G - P^*$ with $j$ red edges. If $C$ has $2j+1$ or more vertices, then it must have at least $2j$ edges, and we are done. Therefore, assume that $C$ has only $2j$ vertices: the endpoints of the $j$ red edges of $C$. 

If a vertex $v$ of $C$ is incident to any edges not in $C$, those edges are in $P^*$, and therefore $v$ must be a vertex of $P^*$. Since all vertices of $C$ are incident to a red edge in $C$, $v$ cannot be incident to any red edge in $P^*$; therefore $v$ is only incident to blue edges in $P^*$, so it must be an anchor.

We classify each red edge of $C$ as
\begin{itemize}
\item \textbf{strong}, if its right endpoint is a right anchor, or if its left endpoint is a left anchor. Let $j_1$ be the number of strong edges in $C$.
\item \textbf{weak}, if neither of its endpoints is an anchor. Let $j_2$ be the number of weak edges in $C$.
\item \textbf{weird}, if its right endpoint is a left anchor, or if its left endpoint is a right anchor. Let $j_3$ be the number of weird edges in $C$.
\end{itemize}

By Claim~\ref{claim:anchors}(i), every red edge in $C$ has at most one anchor; therefore every red edge in $C$ is exactly one of strong, weak, or weird, and $j_1 + j_2 + j_3 = j$. Examples of this classification are shown in Figure~\ref{fig:strong-weak-weird}.

\begin{figure}
	\centering
	\begin{tikzpicture}
		\foreach \i in {1,2,...,8} {
			\node [vtx] (x\i) at (\i,0) {}; \node [below] at (x\i) {$x_{\i}$};
		}
		\node [blue] (xdots) at (9,0) {$\cdots$};
		
		\node [vtx] (y1) at (2.5,1.5) {}; \node [above] at (y1) {$y_1$};
		\node [vtx] (y2) at (4.5,1.5) {}; \node [above] at (y2) {$y_2$};
		\node [vtx] (y3) at (5.5,1.5) {}; \node [above] at (y3) {$y_3$};
		\node [vtx] (y4) at (7.5,1.5) {}; \node [above] at (y4) {$y_4$};
		
		\draw [red] (y1) to node[rlabel]{} (x4) (y2) to node[rlabel]{} (y3) (x6) to node[rlabel]{} (y4);
		\draw [blue] (y1) to node[blabel]{} (y2) (y3) to node[blabel]{} (x6);
		
		\draw [red] (x1) to node[rlabel]{} (x2);
		\draw [blue] (x2) to node[blabel]{} (x3) to node[blabel]{} (x4) to node[blabel]{} (x5) to node[blabel]{} (x6) to node[blabel]{} (x7) to node[blabel]{} (x8) to node[blabel]{} (xdots);
		\draw [decorate,decoration={brace,amplitude=10pt}] (8.3,-0.3) -- (2.7,-0.3);
		\node at (5.5,-1) {Right anchors};
	\end{tikzpicture}

	\caption{Part of path $P^*$ with vertices $x_1, x_2, x_3, \dots$ and part of a component $C$ of $G-P^*$; edge $y_1x_4$ is strong, edge $y_2 y_3$ is weak, and edge $x_6y_4$ is weird}
	\label{fig:strong-weak-weird}
\end{figure}
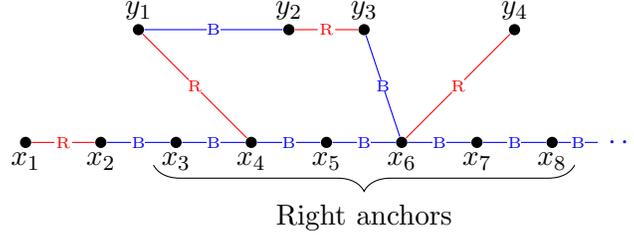

To help prove that there are $j$ blue edges in $C$, we will build up a subgraph $C'$ of $C$. First, we add all vertices of $C$ which are anchors to $C'$. There are exactly $j_1 + j_3$ of these: exactly one endpoint of every strong or weird edge is an anchor, neither endpoint of a weak edge is an anchor, and there are no vertices in $C$ which are not an endpoint of one of its red edges.

Second, we consider the weak edges of $C$. Every weak edge must be part of a $P_n$ in which it is the only red edge; if this $P_n$ is entirely contained in $C$, then $C$ has $n-2 \ge k \ge j$ blue edges, and we are done. Otherwise, the blue $P_n$ must leave $C$; therefore, there is a blue path in $C$ from the weak edge to an anchor. We add this blue path, as well as the weak edge itself, to $C'$.

Third, we consider the weird edges of $C$. Let $vw$ be a weird edge; without loss of generality, $w$ is a left anchor. There must be a $P_n$ containing $vw$ where it is the other blue edge; this decomposes into a blue $P_a$ whose rightmost vertex is $v$, and a blue $P_b$ whose leftmost endpoint is $w$, with $a+b=n$. Because $w$ is not a right anchor, $b < n-k$, and therefore $a \ge k$. However, a blue path in $C$ can only contain one endpoint from each of the $j \le k-1$ red edges, so we can have no more than a $P_{k-1}$ entirely contained in $C$. Therefore the blue $P_a$ must leave $C$: there is a blue path in $C$ from $v$ to an anchor. We add this blue path (but \emph{not} the red edge $vw$) to $C'$. Note that both $v$ and $w$ are now in $C'$; $v$ as part of this blue path, and $w$ as an anchor. The case where $v$ is a right anchor is treated symmetrically.

Now, $C'$ is finalized. Both endpoints of every weak or weird edge of $C$ are in $C'$, contributing $2j_2 + 2j_3$ vertices. Let $i \ge 0$ be the number of non-anchor endpoints of strong edges included in $C'$; since all $j_1$ anchor endpoints of strong edges are included, $C'$ contains $i+j_1+2j_2+2j_3$ vertices. There are at most $j_1+j_3$ connected components in $C'$: we started with the $j_1+j_3$ anchors of $C$, and then added only paths containing at least one of these anchors. Therefore $C'$ has at least $(i+j_1 + 2j_2 + 2j_3) - (j_1 + j_3) = i+2j_2+j_3$ edges. Only $j_2$ edges of $C'$ (the weak edges) are red, so $C'$ has $i+j_2+j_3$ blue edges.

There are $j_1 - i$ non-anchor vertices of strong edges which are not included in $C'$. By Claim~\ref{claim:anchors}(ii) and (iii), each of these must be incident on a blue edge in $C$ (but not in $C'$). Moreover, since (in $G$, and therefore also in $C$) no left anchor is to the left of a right anchor, these edges consist of some number of blue edges going left from a red edge, followed by some number of blue edges going right from a red edge. Therefore these $j_1 - i$ non-anchor vertices must have $j_1 - i$ additional blue edges.

The total number of blue edges is at least $(i+j_2+j_3) + (j_1 - i) = j_1 + j_2 + j_3 = j$. Therefore $C$ contains at least $j$ blue edges, which was what we wanted.
\end{proof}

The lower bound $n \ge \max\{3k-3,2k+1 \}$ might not be optimal. However, Painter's strategy of coloring every edge blue when this does not lose the game assumes that $n$ is large compared to $k$; when $n < k$, Builder can exploit it. 

When $n < k$ and Painter is known to follow this strategy, Builder begins by playing two copies of $P_{n-1}$: one on vertices $10, 20, 30, \dots, 10n-10$ and one on vertices $15, 25, 35, \dots, 10n-5$. Because no blue $P_n$ is created, Painter colors all $2n-4$ edges blue. However, now Builder can play edges 
\[
	\{9,10\}, \{14,15\}, \{20,25\}, \dots, \{10n-20, 10n-15\}, \{10n-10, 10n-9\}, \{10n-5,10n-4\}
\]
and Painter must color each of these red, creating an $M_{n+1}$ in just $3n-3$ moves total: this is $n+2(n+1)-5$. Builder can then play edges $\{10n-20, 10(n+i)\}$ and $\{10(n+i), 10(n+i)+5\}$ for $i=1, \dots, k-n-1$. Painter colors the first edge blue and the second edge red for each $i$, increasing the red matching from $M_{n+1}$ to $M_k$ in $2(k-n-1)$ more moves: $n+2k-5$ moves total. An example where $n=4$ and $k=6$ is illustrated in Figure~\ref{fig:exploit}.

This argument partially justifies the need for a lower bound on $n$ in Theorem~\ref{thm:serial-matching}, though it is possible that Painter has a more refined strategy that does not require it.

\begin{figure}[h!]
	\centering
	\begin{tikzpicture}
		\node [vtx] (9) at (1,0) {}; \node [below] at (9) {$9$};
		\node [vtx] (10) at (2,0) {}; \node [below] at (10) {$10$};
		\node [vtx] (14) at (3,0) {}; \node [below] at (14) {$14$};
		\node [vtx] (15) at (4,0) {}; \node [below] at (15) {$15$};
		\node [vtx] (20) at (5,0) {}; \node [below] at (20) {$20$};
		\node [vtx] (25) at (6,0) {}; \node [below] at (25) {$25$};
		\node [vtx] (30) at (7,0) {}; \node [below] at (30) {$30$};
		\node [vtx] (31) at (8,0) {}; \node [below] at (31) {$31$};
		\node [vtx] (35) at (9,0) {}; \node [below] at (35) {$35$};
		\node [vtx] (36) at (10,0) {}; \node [below] at (36) {$36$};
		\node [vtx] (40) at (11,0) {}; \node [below] at (40) {$40$};
		\node [vtx] (45) at (12,0) {}; \node [below] at (45) {$45$};
		
		\draw [red] (9) to node[rlabel]{} (10) (14) to node[rlabel]{} (15) (20) to node[rlabel]{} (25) (30) to node[rlabel]{} (31) (35) to node[rlabel]{} (36) (40) to node[rlabel]{} (45);
		\draw [blue] (10) to[arch] node[blabel]{} (20) (15) to[arch] node[blabel]{} (25) (20) to[arch] node[blabel]{} (30) (20) to[arch] node[blabel]{} (40) (25) to[arch] node[blabel]{} (35);
	\end{tikzpicture}
	
	\caption{Builder's exploit of Painter's strategy when $n=4$ and $k=6$. Edges $(10,20)$, $(20,30)$, $(15,25)$, $(25,35)$ are played first, then edges $(9,10)$, $(14,15)$, $(20,25)$, $(30,31)$, $(35,36)$, then edge $(20,40)$ and edge $(40,45)$.}
	\label{fig:exploit}
\end{figure}
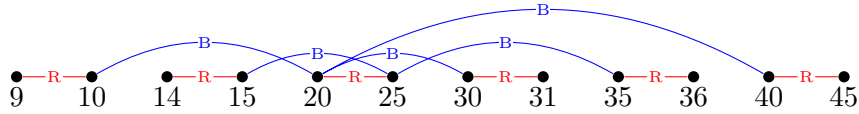

\subsection{The partial St.~Ives matching}

Recall that the partial St.~Ives matching $S_k'$ is obtained from a nested $k$-edge matching by adding $k$ more consecutive nested $k$-edge matchings inside it. In this final section of the paper, we prove Theorem~\ref{thm:partial-st-ives} that for all $k$ there is a constant $c_k$ such that $r_o(S_k',P_n) \le n + c_k$. Throughout this proof, we write $r(n_1, n_2)$ for the Ramsey number of two cliques, $r(K_{n_1}, K_{n_2})$; note that this is the same as $r_o(K_{n_1}, K_{n_2})$, since all vertex orderings of a clique are isomorphic as ordered graphs.

For $0\leq b\leq (k+1)k$, we define a family of $b$-edge subgraphs of $S_k'$ called $\mathcal{S}_k'[b]$. When $b \leq 2k$, the family $\mathcal{S}_k'[b]$ includes only one graph: a nested matching of size $b$, whose inmost edge is denoted by $e_b$. When $b > 2k$,  the family $\mathcal{S}_k'[b]$ includes all ordered graphs consisting of one nested matching of size $k$ and $\lceil \frac bk \rceil - 1$ further nested matchings inside it. These $\lceil \frac bk \rceil - 1$ matchings are consecutive and disjoint; $\lfloor \frac bk \rfloor - 1$ of them have size $k$ and, if $k \nmid b$, one of them has size $b \bmod k$. If $k \mid b$, we denote the set of inmost edges of the $ \frac bk  - 1$ many matchings of size $k$ by $E_b$. If $k \nmid b$, we denote the inmost edge of the matching of size $b \bmod k$ by $e_b$ and set $E_b:=\{e_b\}$.

Note that when $b>2k$, the family $\mathcal{S}_k'[b]$ includes multiple ordered graphs, since we do not specify the location of the nested matching of smaller size. 

For $a\geq (k+1)^3$, we say we are in state $(a,0)$ if there exists a blue $P_a$.

For $b\in [(k+1)^2]\setminus\{2k,3k,\ldots, k^2\}$ and $a\geq (k+1)^3-bk$, we say that we are in state $(a,b)$ if there exists a blue $P_a$ and a red $S\in \mathcal{S}_k'[b]$ where the last $(k+1)^3-bk$ vertices of the blue path are between the two endpoints of the edge $e_b$.

For $b\in \{2k,3k,\ldots, k^2\}$ and $a\geq (k+1)^3-bk$, we say that we are in state $(a,b,0)$ if there exists a blue $P_a$ and a red $S\in \mathcal{S}_k'[b]$ where the last $(k+1)^3-bk$ vertices of the blue path are between the two endpoints of some edge $e_b\in E_b$.

For $b\in \{2k,3k,\ldots, k^2\}$ and $a\geq (k+1)^3-bk$, we say that we are in state $(a,b,1)$ if there exists a blue $P_a$ and a red $S\in \mathcal{S}_k'[b]$ where the last $(k+1)^3-bk$ vertices of the blue path are all between the endpoints of the most inside edge of the outside nested matching, and also either before or after all of the inside copies of the red nested matchings in $S$, or between two consecutive copies of nested matchings in $S$.

We will show that either in one move we increase the first coordinate of our states by 1 or in a constant number of moves the second coordinate by 1 while the first coordinate only decreases by at most a constant. Builder wins when a state $(n',b),(n',b,0), (n',b,1) $ for some $n'\geq n$ and $b\geq 0$ or a state $(a,(k+1)k)$ for some $a\geq 0$ is reached. 

Let $v_a$ be the last vertex in the blue $P_a$ of the current state we are in. 

\begin{claim}
\label{partialstivesc1}
There exists $c_k>0$ such that Builder can reach state $((k+1)^3,0)$ in at most $c_k$ moves.  
\end{claim}
\begin{proof}
Builder simply plays all edges of a clique of size $r((k+1)^3,2(k+1)^2)$ which is a constant depending only on $k$. Then either there exists a red clique of size $2(k+1)^2$, and therefore in particular a red $S\in\mathcal{S}_k'$, or a blue clique of size $(k+1)^3$, and therefore in particular a blue $P_{(k+1)^3}$.
\end{proof}

\begin{claim}
\label{partialstivesc2}
There exists $c_k>0$ such that the following holds.
\begin{itemize}
\item If we are in state $(a,b)$ with $a\geq (k+1)^3-bk$ and $b\not\in\{2k-1,2k,3k-1,3k,\ldots, k^2-1,k^2\}$ we reach state $(a+1,b)$ in 1 move or $(a',b+1)$ in at most $c_k$ moves for some $a' \geq  a$.
\item If we are in state $(a,b)$ with $a\geq (k+1)^3-bk$ and $b\in\{2k-1,3k-1,\ldots, k^2-1\}$ we reach state $(a+1,b)$ in 1 move or $(a',b+1,0)$ in at most $c_k$ moves for some $a' \geq  a$.
\item If we are in state $(a,b,1)$  with $a\geq (k+1)^3-bk$ and $b\in\{2k,3k,\ldots, k^2\}$ we reach state $(a+1,b,1)$ in 1 move or $(a',b+1)$ in at most $c_k$ moves for some $a' \geq  a$.
\end{itemize}
\end{claim} 
\begin{proof} 
Assume we are in the state $(a,b)$ with $a\geq (k+1)^3-bk$ and $b\not\in\{2k,3k,\ldots,k^2\}$, or  the state  $(a,b,1)$  with $a\geq (k+1)^3-bk$ and $b\in\{2k,3k,\ldots, k^2\}$. We start by defining an interval $[u_b,w_b]$ in which all new edges for the following rounds will be played. 

If $b\geq 1$, there exists a blue $P_a$ and a red $S\in \mathcal{S}_k'[b]$ satisfying the properties of state $(a,b)$ or $(a,b,1)$ we start in. We denote by $u_b$ the first vertex to the left of the last $(k+1)^3$ vertices of $P_a$ which is incident to the red $S$. Further, we define $w_b$ to be the first vertex to the right of the last vertex of $P_a$ which is incident to the red $S$.

If $b=0$, there exists a blue $P_a$. Denote by $u_b$ the vertex one to the left of the $(k+1)^3$ last vertex in $P_a$, and denote by $w_b$ an arbitrary vertex to the right of the last vertex in $P_a$. 

Recall that $v_a$ is the last vertex in the blue $P_a$ and note that $u_b < v_a < w_b$. In this case Builder plays an edge from $v_a$ to a vertex $w$ such that $u_b<v_a<w<w_b$. 

First, assume $v_aw$ is painted blue. If if $b\not\in\{2k,3k,\ldots,k^2\}$ we reached state $(a+1,b)$ and if $b\in\{2k,3k,\ldots,k^2\}$ we reached state $(a+1,b,1)$ in 1 move. Next, assume $v_aw$ is painted red. Builder plays all edges of a clique of size $r((k+1)^3,2(k+1)^2)$ where all the vertices of this clique are between $v_a$ and $w$. Then either there exists a red clique of size $2(k+1)^2$ or a blue clique of size $(k+1)^3$ and therefore in a particular a blue $P_{(k+1)^3}$. In the first case, Builder created a red $S_k'$ and therefore wins in a total of at most $a+c_k$ moves for some constant $c_k>0$. 
    
Now assume Builder created a blue path of length $(k+1)^3$. Let $z'$ be the first vertex in this blue path of length $(k+1)^3$ and $z$ be the $((k+1)^3-bk)$-last vertex in the blue $P_a$. Now, Builder plays the edge $zz'$. 
\begin{itemize}
    \item If $zz'$ is painted blue, Builder has created a blue path of length $a'=a-((k+1)^3-bk)+(k+1)^3\geq a$ and a red $S_k'[b+1]$ (with $v_aw$ being the new red edge $e_{b+1}\in E_{b+1}$) where the last $(k+1)^3$ vertices of the blue path are between the two endpoints of the edge $e_{b+1}$. Therefore, if $b\not\in\{2k-1,3k-1,\ldots,k^2-1\}$ we reached state $(a',b+1)$ and if $b\in\{2k-1,3k-1,\ldots,k^2-1\}$ we reached state $(a',b+1,0)$ for $a'\geq a$ in a constant number of moves. See Figure~\ref{fig:zz'blue} for an illustration in the case the edge $zz'$ is painted blue.
    \item If $zz'$ is painted red, there is a blue path of length $a$ and a red $S_k'[b+1]$ (with edge $e_{b+1} 
 \in E_{b+1}$ being the edge $zz'$) where the last $(k+1)^3-bk-1\geq (k+1)^3-(b+1)k$ vertices of the blue $P_a$ are between the two endpoints of the edge $e_{b+1}=zz'$. Thus, if $b\not\in\{2k-1,3k-1,\ldots,k^2-1\}$ we reached state $(a,b+1)$ and if $b\in\{2k-1,3k-1,\ldots,k^2-1\}$ we reached state $(a,b+1,0)$ in a constant number of moves. See See Figure~\ref{fig:zz'red} for an illustration in the case the edge $zz'$ is painted red.
 \qedhere
\end{itemize}
\end{proof}

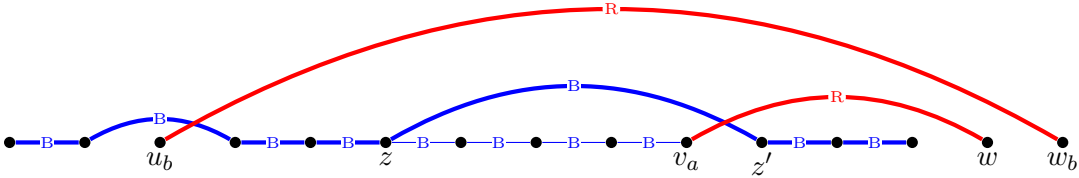
\begin{figure}[h!]
	\centering

	\begin{tikzpicture}
		\node [vtx] (9) at (-2,0) {}; 
		\node [vtx] (8) at (-1,0) {}; 
		\node [vtx] (ub) at (0,0) {}; \node [below] at (ub) {$u_b$};
		\node [vtx] (1) at (1,0) {};
		\node [vtx] (2) at (2,0) {}; 
		\node [vtx] (z) at (3,0) {}; \node [below] at (z) {$z$};
		\node [vtx] (3) at (4,0) {}; 
		\node [vtx] (4) at (5,0) {}; 
		\node [vtx] (5) at (6,0) {}; 
		\node [vtx] (va) at (7,0) {}; \node [below] at (va) {$v_a$};
		\node [vtx] (z') at (8,0) {}; \node [below] at (z') {$z'$};
		\node [vtx] (6) at (9,0) {};
		\node [vtx] (7) at (10,0) {};
		\node [vtx] (w) at (11,0) {}; \node [below] at (w) {$w$};
		\node [vtx] (wb) at (12 ,0) {}; \node [below] at (wb) {$w_b$};
		
		\draw[blue,ultra thick] (9) to node[blabel]{} (8);
		\draw[blue,ultra thick] (8) to[arch] node[blabel]{} (1);
		\draw[blue,ultra thick] (1) to node[blabel]{} (2);
        \draw[blue,ultra thick] (2) to node[blabel]{} (z);
        \draw[blue] (z) to node[blabel]{} (3);
        \draw[blue] (3) to node[blabel]{} (4);
        \draw[blue] (4) to node[blabel]{} (5);
        \draw[blue] (5) to node[blabel]{} (va);
		\draw[blue,ultra thick] (z) to[arch] node[blabel]{} (z');
        \draw[blue,ultra thick] (z') to node[blabel]{} (6);
        \draw[blue,ultra thick] (6) to node[blabel]{} (7);
		\draw[red,ultra thick] (va) to[arch] node[rlabel]{} (w);
        \draw[red,ultra thick] (ub) to[arch] node[rlabel]{} (wb);
	\end{tikzpicture}

	\caption{The situation on the board when $zz'$ is colored blue. The edges contributing to the new state are in bold.}
	\label{fig:zz'blue}
\end{figure}

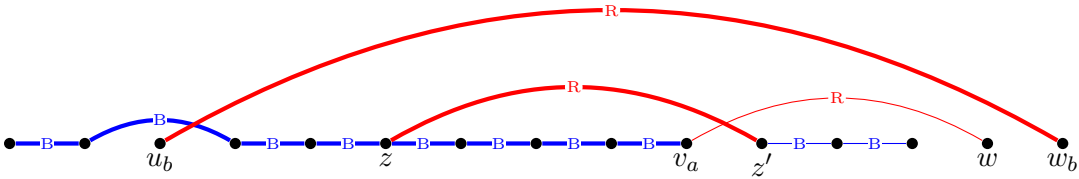
\begin{figure}[h!]
	\centering

	\begin{tikzpicture}
		\node [vtx] (9) at (-2,0) {}; 
		\node [vtx] (8) at (-1,0) {}; 
		\node [vtx] (ub) at (0,0) {}; \node [below] at (ub) {$u_b$};
		\node [vtx] (1) at (1,0) {};
		\node [vtx] (2) at (2,0) {}; 
		\node [vtx] (z) at (3,0) {}; \node [below] at (z) {$z$};
		\node [vtx] (3) at (4,0) {}; 
		\node [vtx] (4) at (5,0) {}; 
		\node [vtx] (5) at (6,0) {}; 
		\node [vtx] (va) at (7,0) {}; \node [below] at (va) {$v_a$};
		\node [vtx] (z') at (8,0) {}; \node [below] at (z') {$z'$};
		\node [vtx] (6) at (9,0) {};
		\node [vtx] (7) at (10,0) {};
		\node [vtx] (w) at (11,0) {}; \node [below] at (w) {$w$};
		\node [vtx] (wb) at (12 ,0) {}; \node [below] at (wb) {$w_b$};
		
		\draw[blue,ultra thick] (9) to node[blabel]{} (8);
		\draw[blue,ultra thick] (8) to[arch] node[blabel]{} (1);
		\draw[blue,ultra thick] (1) to node[blabel]{} (2);
        \draw[blue,ultra thick] (2) to node[blabel]{} (z);
        \draw[blue,ultra thick] (z) to node[blabel]{} (3);
        \draw[blue,ultra thick] (3) to node[blabel]{} (4);
        \draw[blue,ultra thick] (4) to node[blabel]{} (5);
        \draw[blue,ultra thick] (5) to node[blabel]{} (va);
		\draw[red,ultra thick] (z) to[arch] node[rlabel]{} (z');
        \draw[blue] (z') to node[blabel]{} (6);
        \draw[blue] (6) to node[blabel]{} (7);
		\draw[red] (va) to[arch] node[rlabel]{} (w);
        \draw[red,ultra thick] (ub) to[arch] node[rlabel]{} (wb);
	\end{tikzpicture}

	\caption{The situation on the board when $zz'$ is colored red. The edges contributing to the new state are in bold.}
	\label{fig:zz'red}
\end{figure}

\begin{claim}
\label{partialstivesc3}
There exists $c_k>0$ such that the following holds. If we are in state $(a,b,0)$ with $a\geq (k+1)^3-bk$ and $b\in\{2k,3k,\ldots k^2\}$, we reach state $(a',b,1)$ for $a'\geq a-k$ in at most $c_k$ moves.
\end{claim} 
\begin{proof}
Since we are in state $(a,b,0)$ there exists a blue $P_a$ and a red $S\in \mathcal{S}_k'[b]$ where the last $(k+1)^3-bk$ vertices of the blue path are between the two endpoints of some red edge $e_b\in E_b$. Let $u_a$ be the rightmost vertex incident to the red nested matching of size $k$ containing $e_b$, and let $u_b$ be the leftmost vertex, after $u_a$, which is incident to a vertex of $S$. 

Builder plays all edges of a clique of size $r((k+1)^3+k,2(k+1)^2)$ between $u_a$ and $u_b$. Then either Builder created a red $K_{2(k+1)^2}$ or a blue path of length $(k+1)^3+k$. In the first case, Builder created a red $S_k'$ and therefore wins in a total of at most $a+c_k$ moves for some constant $c_k>0$.  
    
Now assume Builder created a blue path of length $(k+1)^3+k$. 

Denote by $c_1',c_2',\ldots, c_k'$ the $k$ first vertices of this blue path and $c_k,c_{k-1},\ldots, c_1$ the last $k$ vertices of the blue $P_a$ (from left to right). Builder now plays the edges $c_1c_1',c_2c_2',\ldots, c_kc_k'$. Note that these $k$ edges form a nested $k$-edge matching.

\begin{itemize}
    \item
If one of those $k$ edges is painted blue, we obtain a blue path of length $a'$ where $a'\geq a$ such that the last $(k+1)^3$ vertices of this blue path are between $u_a$ and $u_b$. We have reached state $(a',b,1)$ in at most $c_k$ moves for some constant $c_k>0$. See Figure~\ref{fig:cc'blue} for an illustration in this case.
\item If all of those $k$ edges are painted red, we obtain a copy $S$ of some $S_k'[b]$ and a blue $P_{a'}$ for $a'=a-k$ with all vertices appearing before the newly created nested matching of size $k$. Moreover, the last $(k+1)^3-bk-k=(k+1)^3-(b+1)k$ vertices of the blue $P_a'$  are all between the endpoints of the most inside edge of the outside nested matching, and also either before or after all of the inside copies of the red nested matchings in $S$, or between two consecutive copies of nested matchings in $S$. 
Thus, we have reached state $(a',b,1)$ in at most $c_k$ moves for some constant $c_k>0$. See Figure~\ref{fig:cc'blue} for an illustration in this case.\qedhere
\end{itemize}
\end{proof}

The proof of Theorem~\ref{thm:partial-st-ives} follows by combining Claims~\ref{partialstivesc1}, \ref{partialstivesc2} and \ref{partialstivesc3}.

\begin{figure}[h!]
	\centering

   \begin{tikzpicture}[scale=0.8]
		\node [vtx] (1) at (-2,0) {}; 
		\node [vtx] (l2) at (-1,0) {}; 
		\node [vtx] (2) at (0,0) {}; 
		\node [vtx] (3) at (1,0) {};
		\node [vtx] (l3) at (2,0) {}; 
		\node [vtx] (3') at (3,0) {}; 
		\node [vtx] (l4) at (4,0) {}; 
		\node [vtx] (4) at (5,0) {}; 
		\node [vtx] (5) at (6,0) {}; 
		\node [vtx] (c3) at (7,0) {}; \node [below] at (c3) {$c_3$};
		\node [vtx] (c2) at (8,0) {}; \node [below] at (c2) {$c_2$};
		\node [vtx] (c1) at (9,0) {}; \node [below] at (c1) {$c_1$};
		\node [vtx] (r1) at (10,0) {};
		\node [vtx] (r2) at (11,0) {}; 
		\node [vtx] (ua) at (12 ,0) {}; \node [below] at (ua) {$u_a$};
		\node [vtx] (c1') at (13 ,0) {}; \node [below] at (c1') {$c_1'$};
		\node [vtx] (c2') at (14,0) {}; \node [below] at (c2') {$c_2'$};
		\node [vtx] (c3') at (15,0) {}; \node [below] at (c3') {$c_3'$};
        \node [vtx] (6) at (16 ,0) {}; 
		\node [vtx] (7) at (17,0) {};
		\node [vtx] (ub) at (18,0) {}; \node [below] at (ub) {$u_b$};
  
		\draw[blue,ultra thick] (1) to[arch] node[blabel]{} (2);
		\draw[blue,ultra thick] (2) to node[blabel]{} (3);
		\draw[blue,ultra thick] (3) to[arch] node[blabel]{} (3');
		\draw[blue,ultra thick] (3') to[arch] node[blabel]{} (4);
        \draw[blue,ultra thick] (4) to node[blabel]{} (5);
		\draw[blue,ultra thick] (5) to node[blabel]{} (c3);
        \draw[blue,ultra thick] (c3) to node[blabel]{} (c2);
        \draw[blue] (c2) to node[blabel]{} (c1);
		\draw[red,ultra thick] (l2) to[arch] node[rlabel]{} (ua);
        \draw[red,ultra thick] (l4) to[arch] node[rlabel]{} (r1);
        \draw[red,ultra thick] (l3) to[arch] node[rlabel]{} (r2);
        \draw[blue] (c1') to node[blabel]{} (c2');
        \draw[blue,ultra thick] (c2') to node[blabel]{} (c3');  
        \draw[blue,ultra thick] (c3') to node[blabel]{} (6); 
        \draw[blue,ultra thick] (6) to node[blabel]{} (7);
        \draw[blue,ultra thick] (c2) to[arch] node[blabel]{} (c2');
        \draw[red,ultra thick] (ub) to  node[rlabel]{} (18,1);
	\end{tikzpicture}
 
 \caption{The situation on the board when one of the edges $c_1c_1',c_2c_2',\ldots, c_kc_k'$ is colored blue (for $k=3$). The edges contributing to the new state are in bold. Depending on the situation on the board, the red edge leaving $u_b$ might go to a vertex to the left or right of all of the displayed vertices.}
	\label{fig:cc'blue}
\end{figure}
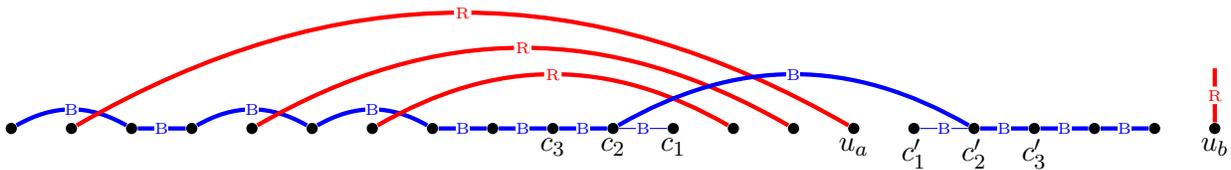

\begin{figure}[h!]
	\centering

     \begin{tikzpicture}[scale=0.8]
		\node [vtx] (1) at (-2,0) {}; 
		\node [vtx] (l2) at (-1,0) {}; 
		\node [vtx] (2) at (0,0) {}; 
		\node [vtx] (3) at (1,0) {};
		\node [vtx] (l3) at (2,0) {}; 
		\node [vtx] (3') at (3,0) {}; 
		\node [vtx] (l4) at (4,0) {}; 
		\node [vtx] (4) at (5,0) {}; 
		\node [vtx] (5) at (6,0) {}; 
		\node [vtx] (c3) at (7,0) {}; \node [below] at (c3) {$c_3$};
		\node [vtx] (c2) at (8,0) {}; \node [below] at (c2) {$c_2$};
		\node [vtx] (c1) at (9,0) {}; \node [below] at (c1) {$c_1$};
		\node [vtx] (r1) at (10,0) {};
		\node [vtx] (r2) at (11,0) {}; 
		\node [vtx] (ua) at (12 ,0) {}; \node [below] at (ua) {$u_a$};
		\node [vtx] (c1') at (13 ,0) {}; \node [below] at (c1') {$c_1'$};
		\node [vtx] (c2') at (14,0) {}; \node [below] at (c2') {$c_2'$};
		\node [vtx] (c3') at (15,0) {}; \node [below] at (c3') {$c_3'$};
        \node [vtx] (6) at (16 ,0) {}; 
		\node [vtx] (7) at (17,0) {};
		\node [vtx] (ub) at (18,0) {}; \node [below] at (ub) {$u_b$};
  
		\draw[blue,ultra thick] (1) to[arch] node[blabel]{} (2);
		\draw[blue,ultra thick] (2) to node[blabel]{} (3);
		\draw[blue,ultra thick] (3) to[arch] node[blabel]{} (3');
		\draw[blue,ultra thick] (3') to[arch] node[blabel]{} (4);
        \draw[blue,ultra thick] (4) to node[blabel]{} (5);
		\draw[blue] (5) to node[blabel]{} (c3);
        \draw[blue] (c3) to node[blabel]{} (c2);
        \draw[blue] (c2) to node[blabel]{} (c1);
		\draw[red] (l2) to[arch] node[rlabel]{} (ua);
        \draw[red] (l4) to[arch] node[rlabel]{} (r1);
        \draw[red] (l3) to[arch] node[rlabel]{} (r2);
        \draw[blue] (c1') to node[blabel]{} (c2');
        \draw[blue] (c2') to node[blabel]{} (c3');  
        \draw[blue] (c3') to node[blabel]{} (6); 
        \draw[blue] (6) to node[blabel]{} (7);
        \draw[red,ultra thick] (c1) to[arch] node[rlabel]{} (c1');
        \draw[red,ultra thick] (c3) to[arch] node[rlabel]{} (c3');
        \draw[red,ultra thick] (c2) to[arch] node[rlabel]{} (c2');
        \draw[red,ultra thick] (ub) to  node[rlabel]{} (18,1);
	\end{tikzpicture}

    \caption{The situation on the board when all of the edges $c_1c_1',c_2c_2',\ldots, c_kc_k'$ are colored red (for $k=3$). The edges contributing to the new state are in bold. Depending on the situation on the board, the red edge leaving $u_b$ might go to a vertex to the left or right of all of the displayed vertices.}
	\label{fig:cc'red}
\end{figure}
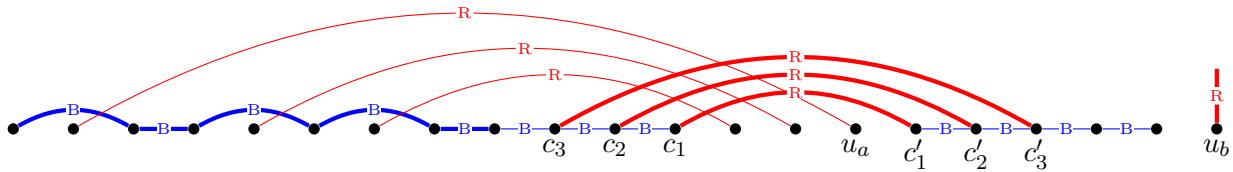

\section*{Acknowledgements}  We thank the referees for their careful reading of the manuscript and their valuable suggestions, in particular for suggesting a strategy for Builder improving the upper bound of Lemma~\ref{lemma:XPn}.

\bibliographystyle{abbrv}
\bibliography{orderedpath}

\end{document}